\documentclass[12pt]{article}
\usepackage{color}
\usepackage{amsmath}

\usepackage{extarrows}
\usepackage{geometry}
\usepackage{authblk} 
\usepackage{hyperref} 
\usepackage{color}
\usepackage{ntheorem}
\usepackage{romannum}
\usepackage[mathscr]{euscript}
\usepackage[utf8]{inputenc}
\usepackage[bottom]{footmisc}
\def\0{\emptyset}

\newtheorem{theorem}{Theorem}[section]
\newtheorem{definition}[theorem]{Definition}
\newtheorem{lemma}[theorem]{Lemma}
\newtheorem{claim}[theorem]{Claim}

\newtheorem{conjecture}[theorem]{Conjecture}
\newtheorem{cor}[theorem]{Corollary}
\newtheorem{proposition}[theorem]{Proposition}

\newenvironment{proof}{{\noindent\it Proof.}}{\hfill $\square$\par}

\usepackage{graphicx}

\usepackage{enumerate}
\usepackage{enumitem}

\usepackage{
	amsmath,			
	amssymb,			
	enumerate,		    
	graphicx,			
	lastpage,			
	multicol,			
	multirow,			
	pifont,			    
}

\usepackage[numbers]{natbib}

\newcommand{\lf}{\left\lfloor}
\newcommand{\rf}{\right\rfloor}

\newcounter{cases}
\newcounter{subcases}[cases]
\newenvironment{mycase}
{
    \setcounter{cases}{0}
    \setcounter{subcases}{0}
    \newcommand{\case}
    {
        \par\indent\stepcounter{cases}\textbf{Case \thecases.}
    }
    
}
{
    \par
}
\renewcommand*\thecases{\arabic{cases}}

\begin{document}

\pagenumbering{arabic}

\title{Counting induced subgraphs with given intersection sizes}


\author{\small\bf Haixiang Zhang\thanks{email: zhang-hx22@mails.tsinghua.edu.cn}}
\author{\small\bf Yichen Wang\thanks{Corresponding author: email: wangyich22@mails.tsinghua.edu.cn}}
\author{\small\bf Xiamiao Zhao\thanks{email: zxm23@mails.tsinghua.edu.cn}}
\author{\small\bf Mei Lu\thanks{email: lumei@tsinghua.edu.cn}}

\affil{\small Department of Mathematical Sciences, Tsinghua University, Beijing 100084, China}

\date{}

\maketitle\baselineskip 16.3pt

\begin{abstract}
Let $F$ be a graph of order $r$. In this paper, we study the maximum number of induced copies of  $F$ with restricted intersections, which highlights the motivation from extremal set theory. Let $L=\{\ell_1,\dots,\ell_s\}\subseteq[0,r-1]$ be an integer set with $s\not\in\{1,r\}$. Let $\Psi_r(n,F,L)$ be the maximum number of induced copies of $F$ in an $n$-vertex graph, where  the induced copies of $F$ are $L$-intersecting as a family of $r$-subsets, i.e., for any two  induced copies of $F$, the size of their intersection is in $L$. Helliar and Liu initiated a study of the function $\Psi_r(n,K_r,L)$. Very recently, Zhao and Zhang improved their result and showed that $\Psi_r(n,K_r,L)=\Theta_{r,L}(n^{s})$ if and only if $\ell_1,\dots,\ell_s,r$ form an arithmetic progression. In this paper, we show that  $\Psi_r(n,F,L)=o_{r,L}(n^{s})$ when $\ell_1,\dots,\ell_s,r$ do not form an arithmetic progression. We  study the asymptotical result of $\Psi_r(n,C_r,L)$, and determined the asymptotically optimal result when $\ell_1,\dots,\ell_s,r$ form an arithmetic progression and take certain values. We also study the generalized Tur\'an problem, determining the maximum number of $H$, where the copies of $H$ are $L$-intersecting as a family of $r$-subsets. The entropy method is used to prove our results.
\end{abstract}


\vskip.3cm

\section{Introduction}

In this paper, we use capital letters to represent graphs, such as $G, F$, and we use cursive letters to represent hypergraphs, such as $\mathscr{H}$. We use $V(G)$~(resp $V(\mathscr{H})$) to represent the vertex set of $G$~(resp $\mathscr{H}$), and $E(G)$~(resp $E(\mathscr{H})$) to represent the edge set of $G$~(resp $\mathscr{H}$). We use $v(\mathscr{H})$ and $|\mathscr{H}|$ to represent the order and size of $\mathscr{H}$. For convenience, we use $A \in \mathscr{H}$ to represent $A \in E(\mathscr{H})$.
For a vertex subset $A \subseteq V(G)$, we use $G[A]$ to denote the subgraph of $G$ induced by $A$.
For a graph $G$ (resp  hypergraph $\mathscr{H}$) and $v\in V(G)$ (resp $v\in V(\mathscr{H})$), let $d_G(v)$ (resp $d_\mathscr{H}(v)$) be the number of edges (resp hyperedges) contained $v$.
When there is no ambiguity, we may ignore the subscript $G$ or $\mathscr{H}$.
We write $G_1 \cong G_2$, if there is an bijection $\phi:V(G_1) \rightarrow V(G_2)$, such that for any $u,v\in V(G_1)$, $uv \in E(G_1)$ if and only if $\phi(u)\phi(v) \in E(G_2)$. 

The general Tur\'an problem is to determine the maximum number of copies of $F$ on an $n$-vertex $\mathscr{G}$-free graph, where $\mathscr{G}$ is a family of graphs. When $F$ is an edge, it is the celebrated Tur\'an problem \cite{turan1941extremal}.

Let $G$ be a graph on $n$ vertices and $F$ be a fixed graph. The number of induced copies of $F$ in $G$ is the number of subsets $S$ of $V(G)$ such that $G[S]\cong F$, which is denoted as $N_{ind}(G,F)$. The maximum number of induced copies of $F$ on $n$-vertex graph is denoted as $i(n,F):=\max\{ N_{ind}(G,F):~v(G)=n\}.$ And the {\em inducibility} of $F$ is denoted as $ind(F):=\lim_{n\to \infty}\frac{i(n,F)}{ \binom{n}{v(F)}}.$

The first study about inducibility is back to Pippenger and Golumbic \cite{PIPPENGER1975189} in 1975, who proved the lower bound of $ind(F)$ for a $k$-vertex graph $F$ is at least $\frac{k!}{k^k-k}$, and they conjecture that when $F$ is a cycle on $k$ vertices, this bound is tight, i.e. $ind(C_k)=\frac{k!}{k^k-k}$ when $k\geq 5$.
They also gave an upper bound about $ind(C_k)$, which is $(2e+o(1))\frac{k!}{k^k}$. After a series of studies \cite{hefetz2018inducibility}, the optimal upper bound is $(2+o(1))\frac{k!}{k^k}$ given by Kr\'al', Norin and Volce \cite{norin2019bound} by probability method.

Recently, thanks to the celebrated flag algebra method provided by Razborov \cite{razborov2007flag}, there are some important results about the inducibility of small graphs. Balogh, J{\'o}zsef, Hu and Lidick\'y \cite{balogh2016maximum} gave the inducibility of $C_5$. Hirst and James \cite{hirst2014inducibility} calculate some inducibility of $4$-vertex graphs.

This paper combines the general Tur\'an problem with inducibility and extremal set theory. Let us begin with some basic definitions. Denote  $[n]=\{1,2,\dots,n\}$. For any two integers $a<b$, let $[a,b]=\{a,a+1,\dots,b\}$. Write $\binom{X}{r}$  the set of all $r$-subsets of $X$. 

A hypergraph $\mathscr{H}$ is {\em $r$-uniform} if $E(\mathscr{H}) \subseteq \binom{V(\mathscr{H})}{r}$.
Let $L$ be a set of non-negative integers.
An $r$-uniform hypergraph (also call $r$-graph) $\mathscr{H}$ is {\em $L$-intersecting} if the intersection size of any two hyperedges belongs to $L$.
When $L=[t,r-1]$ for some $0\leq t \leq r-1$, an $L$-intersecting $r$-graph is simply said to be $t$-intersecting. Let $\Phi_r(n, L)$ denote the maximum possible size of an $n$-vertex $L$-intersecting $r$-graph, that is,
\[ \Phi_r(n,L) = \max\{|\mathscr{H}| : \text{$\mathscr{H}$ is an $L$-intersecting $r$-graph with $n$ vertices}\},\]
which has been extensively studied in extremal set theory \cite{frankl2016invitation}. The following celebrated Deza-Erd\H{o}s-Frankl \cite{deza1978intersection} theorem provides a general upper bound on $\Phi_r(n,L).$
\begin{theorem}[Deza-Erd\H{o}s-Frankl \cite{deza1978intersection}.]\label{thm: deza-E-F}
    Let $r\geq 3$, $n\geq 2^rr^3$ and $L\subseteq [0,r-1]$. If $\mathscr{H}$ is an $n$-vertex $L$-intersecting $r$-graph, then $|\mathscr{H}|\leq \prod_{\ell\in L}\frac{n-\ell}{r-\ell}\leq n^{|L|}$. Moreover, there exists a constant $C=C(r,L)$ such that every $n$-vertex $L$-intersecting $r$-graph $\mathscr{H}$ with $|\mathscr{H}|\geq Cn^{|L|-1}$ satisfies $\left|\bigcap_{A\in \mathscr{H}} A \right| \geq \min L$.
\end{theorem}

 For $r\geq 3$, given  $G$ and $F$ with $v(F)=r$, define their associated $r$-graph to be $\mathscr{H}^r_{G,F}:=\{S\in \binom{V(G)}{r}: G[S]\cong F\}$. For an integer set $L\subseteq [0,r-1]$, a graph $G$ is called \textbf{$(F, L)$-intersecting} if its associated $r$-graph $\mathscr{H}^r_{G,F}$ is $L$-intersecting, in another word, the intersecting size of any two induced copies of $F$ in $G$ is in $L$. For a fixed set of vertices $W \subseteq V(G)$ with size $\ell$, we define the associated $(r-\ell)$-graph to be $\mathscr{H}^{r-\ell}_{G,F,W}:=\{S\in \binom{V(G)\setminus W}{r -\ell}:~G[S\cup W]\cong F\}$. Note that if $\mathscr{H}^r_{G,F}$ is $L$-intersecting, then $\mathscr{H}^{r-\ell}_{G,F,W}$ is $L'$-intersecting where $L'=\{x-\ell \mid x \in L, x \ge \ell \}$. Moreover, if $W \subseteq \bigcap_{A \in \mathscr{H}_{G,F}^{r}}$, then $|\mathscr{H}^{r-\ell}_{G,F,W}| = |\mathscr{H}^r_{G,F}|$.

Then we define the following function:
$$\Psi_r(n,F,L):=\max\{N_{ind}(G,F)~|~  \text{$G$ is an $n$-vertex $(F,L)$-intersecting graph}\}.$$
Meanwhile, we define a similar extremal function when all the induced copies have a common intersecting set of vertices with size at least $\ell$, that is,
\begin{equation*}
\begin{aligned}
    &\Psi_r(n,F,L,\ell ) :=  \max\left\{N_{ind}(G,F)~|~\text{$G$ is $(F,L)$-intersecting,~} v(G)=n,\left|\bigcap_{A \in E(\mathscr{H}_{G,F}^{r})}A \right| \ge \ell  \right\}.
\end{aligned}
\end{equation*}

Obviously from the definition, we have
\[
\Psi_r(n,F,L,\ell ) \le \Psi_r(n,F,L) \le \Phi_r(n,L).
\]
Recently, Helliar and Liu \cite{helliar2024generalized} investigated the extremal function $\Psi_r(n,K_r,L)$, which can be rephrased as the generalized Tur\'an number that counts the maximum number of $r$-cliques in an $n$-vertex graph without two $r$-cliques whose intersection has size in $[0,r-1]\setminus L$.
Thus, Theorem \ref{thm: deza-E-F} implies that when $n$ is large enough, we have
$$\Psi_r(n,F,L)\leq \prod_{\ell\in L}\frac{n-\ell}{r-\ell}=O(n^{|L|}),$$
for every $F$ with $v(F)=r\geq 3$.
In particular,  the  Ruzsa-Szemer\'edi  theorem \cite{ruzsatriple} is equivalent to the statement that $n^{2-o(1)}<\Psi_3(n,K_3,\{0,1\})=o(n^2)$ if $L=[0,t]$ with $t\in [r-2]$. Recently this result was extended to any fixed $t\in[r-2]$ by Gowers and Janzer \cite{gowers2021generalizations}, that is
\begin{equation}\label{eq: general of (6,3)}
    n^{2-o(1)}<\Psi_r(n,K_r,[0,t])=o(n^{t+1}),
\end{equation}
when $n$ is large enough. Some special case when $L=[0,t-1]\setminus\{s\}$, the values of $\Psi_r(n,K_r,L)$ have been considered in \cite{liu2021generalized,sos1976remarks}. Recently, Helliar and Liu \cite{helliar2024generalized} improved the upper bound of $\Psi_r(n,K_r,L)$ by a constant factor when $2\leq |L|\leq r-1$.
\begin{theorem}[Helliar-Liu \cite{helliar2024generalized}]\label{thm: Helliar with procession}
    Let $r\geq 3$ and let $L\subseteq[0,r-1]$ be a subset with $2\leq |L|\leq r-1$. Then for $n\geq (2r)^{r+1}$,
    $$\Psi_r(n,K_r,L)\leq \left(1-\frac{1}{3r}\right)\prod_{\ell\in L}\frac{n-\ell}{r-\ell}.$$
\end{theorem}
They also gave an optimal result when $L=[t,r-1]$ for $1\leq t<r$.
We set the Tur\'an graph $T(n,t)$ to stand for the balanced complete $t$-partite graph on $n$ vertices.

\begin{theorem}[Helliar-Liu \cite{helliar2024generalized}]
    Let $r>t\geq 1$ be integers. When $n$ is large enough, every $(K_r,t)$-intersecting graph $G$ on $n$ vertices satisfies
    $$N_{ind}(G,K_r)\leq N_{ind}(T(n-t,r-t),K_{r-t})=(1+o(1))\left(\frac{n-t}{r-t} \right)^{r-t},$$
    and equality holds if and only if $G$ is isomorphic to $K_t+T(n-t,r-t)$.
\end{theorem}
Very recently, Zhao and Zhang \cite{zhao2025counting} completely characterizing all parameters $(r,L)$ with $\Psi_r(n,L)=\Theta_{r,L}(n^{|L|})$ and determined the corresponding exact values of $\Psi_r(n,L)$ for large $n$.
\begin{theorem}[Zhao-Zhang \cite{zhao2025counting}]\label{thm: zhao-zhang Kr}
    Let $r\geq 3$ be a fixed integer and let $L=\{\ell_1,\dots,\ell_s\}\subseteq[0,r-1]$ be a fixed subset of size $s\not\in \{1,r\}$ with $\ell_1<\ell_2<\dots<\ell_s$. Then the following statements hold.

    (1) If $\ell_1,\dots,\ell_s, r$ do not form an arithmetic progression, then $\Psi_r(n,K_r,L)=o(n^s)$. If further $r-\ell_s=\ell_s-\ell_{s-1}$, then $\Psi_r(n,K_r,L)=O(n^{s-1})$.

    (2) If $\ell_1,\dots,\ell_s, r$ form an arithmetic progression and let $d:=r-\ell_s=\dots=\ell_2-\ell_1=(r-\ell_1)/s$ be the common difference, then when $n$ is large enough, we have
    $$\Psi_r(n,K_r,L)=N_{ind}(T(\lf(n-\ell_1)/d\rf,s),K_s)=(1+o(1)) \left(\frac{n-\ell_1}{r-\ell_1}\right)^s,$$
    and the extremal graph is unique up to isomorphism.
\end{theorem}
Observed that Theorem \ref{thm: zhao-zhang Kr} (2) is an extension of Theorem \ref{thm: Helliar with procession}.
Notice that when $L=[0,r]$, $\Psi_r(n,F,L)=i(n,F)$, which is equal to determine the inducibility of $F$. Hence, the problem of determining $\Psi_r(n,F,L)$ is a generalization of inducibility.
If $|L|=1$,  by Theorem~\ref{thm: deza-E-F}, the  extremal graph is easily characterized since all induced copies have a common intersection.
In this paper, we consider the case when $|L|\in [2,r-1]$.
We generalize Theorem \ref{thm: zhao-zhang Kr} (1) from $K_r$ to a graph $F$ with $v(F)=r\geq 3$.
\begin{theorem}\label{thm: the main theorem}
    Let $r\geq 3$ be a fixed integer and $F$ be a fixed graph with $v(F)=r$. Let $L=\{\ell_1,\dots,\ell_s\}\subseteq[0,r-1]$ be a fixed subset of size $s\notin \{1,r\}$ with $\ell_1<\ell_2<\dots<\ell_s$. If $\ell_1,\ell_2,\dots, \ell_s,r$ do not form an arithmetic progression, then $\Psi_r(n,F,L)=o(n^{s})$. Moreover, if
         (1)   $r-\ell_s=\ell_s-\ell_{s-1}$, or (2)
        $r-\ell_s\neq \ell_s-\ell_{s-1}$ and $\ell_s-\ell_{s-1}\nmid r-\ell_s$,
       then we have $\Psi_r(n,F,L)=O(n^{s-1})$.
\end{theorem}

We calculate $\Psi_r(n,C_r,L)$ instead of $\Psi_r(n,K_r,L)$ in Theorem \ref{thm: zhao-zhang Kr} (2).
Note that we only need to consider  $r \ge 4$, since $K_3=C_3$.
And we only consider  $0 \notin L$ since if $0 \in L$, the problem will be hard as the inducibility of cycles~(a special case when $L=[0,r-1]$).

\begin{theorem}\label{thm: the main theorem 2}
    Let $r\geq 4$ be a fixed integer and let $L=\{\ell_1,\dots,\ell_s\}\subseteq[1,r-1]$ be a fixed subset of size $s > 1$ with $\ell_1<\ell_2<\dots<\ell_s$.
    If $\ell_1,\ell_2,\dots, \ell_s,r$ form an arithmetic progression and  $d$ is the common difference, then the following statements hold.
    \begin{enumerate}[label=(\arabic*)]
        \item If $d=1$, then
            $$ \Psi_r(n,C_r,L)=\left\{
            \begin{aligned}
            &(1+o(1))\left(\frac{n-\ell_1}{r-\ell_1}\right)^s, \quad~\text{when $(\ell_1,s)\not\in \{(1,3),(2,2)\}$,} \\
            &(2+o(1))\left(\frac{n-\ell_1}{r-\ell_1}\right)^s, \quad~\text{when $(\ell_1,s)\in \{(1,3),(2,2)\}$.}
            \end{aligned}
            \right.
            $$
        \item If $d \ge 2$, then 
        $$ (1+o(1))\left(\frac{n-\ell_1}{r-\ell_1}\right)^s \le \Psi_r(n,C_r,L)\le (2+o(1))\left(\frac{n-\ell_1}{r-\ell_1}\right)^s.$$
        \item If $d\geq 2$ and $\ell_1 > sd$, then
        $$ \Psi_r(n,C_r,L)= (1+o(1))\left(\frac{n-\ell_1}{r-\ell_1}\right)^s.$$
        \item If $d\geq 2$, $s=2$, $\ell_1 < 2d$ and $ \ell_1 \in \{1,2\}$, then
        $$ \Psi_r(n,C_r,L)= (2+o(1))\left(\frac{n-\ell_1}{r-\ell_1}\right)^s.$$
        \item If $d \ge 2, s=2,  \ell_1 < d/2 $ and $\ell_1 \in \{3,4\}$, then
        $$ \Psi_r(n,C_r,L)= (2+o(1))\left(\frac{n-\ell_1}{r-\ell_1}\right)^s.$$
        \item If $d \ge 2, s=2$, and  $ \ell_1 = 2d $, then
        $$ \left(\frac{4}{3} +o(1) \right) \left(\frac{n-\ell_1}{r-\ell_1}\right)^s \le \Psi_r(n, C_r, L) \le \left(2 - \frac{2}{f(\ell_1)} +o(1)\right) \left(\frac{n-\ell_1}{r-\ell_1}\right)^s,$$
        where $f(\ell_1)$ will be defined later.
    \end{enumerate}
\end{theorem}

In the following, we give the definition of $f(\ell_1)$ appeared in Theorem~\ref{thm: the main theorem 2} (6) which has a connection with perfect $1$-factorization.
A \textbf{factor} of a graph $G$ is a spanning subgraph, i.e., a subgraph that has the same vertex set as $G$.
A \textbf{$k$-factor} of $G$ is a spanning $k$-regular subgraph.
A \textbf{$k$-factorization} partitions the edges of $G$ into disjoint $k$-factors.
 $G$ is \textbf{$k$-factorable} if it admits a $k$-factorization.
Particularly, a $1$-factor is a perfect matching, and a $1$-factorization of a $k$-regular graph is a proper edge coloring with $k$ colors.
A \textbf{perfect pair} from a $1$-factorization is a pair of $1$-factors whose union induces a Hamiltonian cycle.
A \textbf{perfect $1$-factorization}~(P1F) of a graph is a $1$-factorization having the property that every pair of $1$-factors is a perfect pair. If we view each $1$-factor~(perfect matching) as a color in the edge coloring, then using the language of edge coloring, a P1F of a $k$-regular graph is an edge coloring with $k$ colors such that
\begin{enumerate}
    \item each color class is a perfect matching,
    \item and every two colors induce a Hamiltonian cycle.
\end{enumerate}

In 1964, Kotzig~\cite{kotzig1963hamilton} first posed this definition and conjectured that every complete graph $K_{2n}$ where $n \ge 2$ has a P1F.
So far, it is proved that some class of graphs have a P1F, such as, $K_{2p}$~\cite{anderson1973finite} and $K_{p+1}$~\cite{wallis1997one,bryant2006new} where $p$ is a prime.
For more details about perfect $1$-factorization, the readers may refer to the survey~\cite{rosa2019perfect}.

We now add a new constraint.
We say a $1$-factorization of a graph $G$ is \textbf{flawless} if it is perfect and every Hamiltonian cycle in graph $G$ is a union of two $1$-factors from the $1$-factorization. Using the language of edge coloring, a flawless $1$-factorization of a $k$-regular graph is an edge coloring with $k$ colors such that
\begin{enumerate}
    \item each color class is a perfect matching,
    \item every two colors induce a Hamiltonian cycle,
    \item and each Hamiltonian cycle only contains two colors.
\end{enumerate}

For an even number $n$, let $f(n)$ be the maximum number of $k$ such that there exists a $k$-regular graph $G$ with $n$ vertices which admits a flawless $1$-factorization.
We will prove $f(n) \ge 3$ for even $n \ge 4$ and we conjecture that
\begin{conjecture}
    $f(n) = 3$ for even $n \ge 4$.
\end{conjecture}
Note that if the conjecture is true, then the equality in Theorem~\ref{thm: the main theorem 2}~(6) holds.
With computer assistance, we can verify that the conjecture for $n=4,6,8$.

Furthermore, we  study the number of copies of $F$ instead of  induced copies of $F$. Denote $N(G, F)$ as the number of copies of $F$ in $G$.   Let $v(F)=r$, and we define $$\mathscr{H}^r_{G,span(F)}:=\left\{S\in \binom{V(G)}{r}:~ G[S] \text{~contains a copy of $F$} \right\}.$$ For an integer set $L\subseteq [0,r-1]$, a graph $G$ is called $(span(F),L)$-intersecting if its associate $r$-graph $\mathscr{H}^r_{G,span(F)}$ is $L$-intersecting.
We rewrite the extremal function as
$$\Psi_r(n,span(F),L):=\max\{N(G,F):~ \text{$G$ is an $n$-vertex $(span(F),L)$-intersecting graph}\}.$$
Similarly, we can define the function $\Psi_r(n,span(F),L,\ell)$ be the maximum number of $N(G,F)$ where $G$ is an $n$-vertex $(span(F),L)$-intersecting graph and all the induced copies of $F$ have a common intersection of size at least $\ell$.
We have a similar result as Theorem \ref{thm: the main theorem}.
\begin{theorem}\label{thm: the main theorem 3}
    Let $r\geq 3$ be a fixed integer and $F$  a fixed graph with $v(F)=r$. Let $L=\{\ell_1,\dots,\ell_s\}\subseteq[0,r-1]$ be a fixed subset of size $s\notin \{1,r\}$ with $\ell_1<\ell_2<\dots<\ell_s$. If $\ell_1,\ell_2,\dots, \ell_s,r$ do not form an arithmetic progression, then $\Psi_r(n,span(F),L)=o(n^{s})$. Moreover, if $r-\ell_s=\ell_s-\ell_{s-1}$ or $r-\ell_s\neq \ell_s-\ell_{s-1}$ and $\ell_s-\ell_{s-1}\nmid r-\ell_s$, then $\Psi_r(n,span(F),L)=O(n^{s-1})$.
\end{theorem}

The paper is organized as follows.
In Section~\ref{sec: th1,3}, we will first prove Theorems~\ref{thm: the main theorem} and \ref{thm: the main theorem 3}.
Then we introduce some tools which will be used to prove Theorem~\ref{thm: the main theorem 2} in Section~\ref{sec: preliminary}.
In Sections~\ref{sec: proof AP 1} to \ref{sec: 6}, we will prove Theorem~\ref{thm: the main theorem 2}.

\section{Proof of Theorems~\ref{thm: the main theorem} and \ref{thm: the main theorem 3}}\label{sec: th1,3}

\subsection{Proof of Theorem~\ref{thm: the main theorem}}

 As in \cite{helliar2024generalized} and  \cite{zhao2025counting}, our strategy is to apply induction on both $r$ and the cardinality of the set $L$. The main idea of the proof comes from \cite{zhao2025counting}.
The key definitions, {\bf good and great induced $F$} which we will introduce later, helps us to generalize the result of $K_r$ to any induced $F$ with $r$ vertices.
In an $r$-uniform hypergraph $\mathscr{H}$, we call a collection of hyperedges $A_1,\dots,A_k$ a \textbf{sunflower} with core $C$ if $\bigcap_{i=1}^kA_i=C$.
We will start with the base case when $L=\{\ell_1,\ell_2\}$ where $0\leq \ell_1<\ell_2<r$.
\begin{lemma}\label{lem: L=l1,l2}
    Let $r\geq 3$, $0\leq \ell_1<\ell_2$ with $r-\ell_2\neq \ell_2-\ell_1$, and $F$ be a graph with $r$ vertices. Then $\Psi_r(n,F,\{\ell_1,\ell_2\})=o(n^2)$. Moreover, if $\ell_2-\ell_1\nmid r-\ell_1$, then $\Psi_r(n,F,\{\ell_1,\ell_2\})=O(n)$.
\end{lemma}

\begin{proof}
   Let $G$ be an $(F,L)$-intersecting graph with $n$ vertices maximizing $N_{ind}(G,F)$, where $L=\{\ell_1,\ell_2\}$.
    Let $\mathscr{H}_{G,F}^r$ be the associated $r$-graph. Then $\mathscr{H}_{G,F}^r$ is $L$-intersecting and we have $|\mathscr{H}^r_{G,F}| = N_{ind}(G,F) = \Psi_r(n,F,L)$.
    Let $C_0=C(r,L)$ be the constant in Theorem \ref{thm: deza-E-F}.   If $|\mathscr{H}_{G,F}^r|\leq C_0n$, then we are done; otherwise, $|\bigcap_{A\in \mathscr{H}^r_{G,F}}A|\geq\min L=\ell_1$ by Theorem \ref{thm: deza-E-F},. Let $W\subseteq\bigcap_{A\in \mathscr{H}^r_{G,F}}A$ such that $|W|=\ell_1$.
    If $\ell_1=0$, then $W=\emptyset$.
    Recall that the hyperedges of $\mathscr{H}^{r-\ell_1}_{G,F,W}$ is the collecting of sets $S\subseteq V(G)$ such that $G[S\cup W]\cong F$. Then $\mathscr{H}^{r-\ell_1}_{G,F,W}$ is $\{0,\ell_2-\ell_1\}$-intersecting.

    As in \cite{zhao2025counting}, we greedily delete vertices with small degree from $G$. We consider the following iterative process. Let $G_0:= G$.
    For $i\geq 0$,  if there exists  $u\in V(G_i)\setminus W$ with $d_{\mathscr{H}^{r-\ell_1}_{G_i,F,W}}(u)<(r-\ell_1)^2$, then delete $u$ from $G_i$ to obtain a new graph $G_{i+1}$; otherwise we stop. Suppose we finally stop at $G_k$.
    Then we have
   \begin{equation}\label{eq-1}
    |\mathscr{H}^{r}_{G,F}|=|\mathscr{H}^{r-\ell_1}_{G,F,W}|\leq |\mathscr{H}^{r-\ell_1}_{G_k,F,W}|+(r-\ell_1)^2k\leq |\mathscr{H}^{r-\ell_1}_{G_k,F,W}|+r^2n.
   \end{equation}

    For each subset $C\subseteq V(G_k)$ of size $\ell_2-\ell_1$, let $\mathcal{S}_C$ denote the maximum sunflower in $\mathscr{H}^{r-\ell_1}_{G_k,F,W}$  with core $C$. As in \cite{helliar2024generalized}, let us define
    \begin{equation*}
    \begin{aligned}
        & U(G_k):=\{u\in V(G_k)\setminus W:~d_{\mathscr{H}^{r-\ell_1}_{G_k,F,W}}(u)\geq (r-\ell_1)^2\},\\
        &\mathcal{C}(G_k):=\left\{C\in \binom{V(G_k)\setminus W}{\ell_2-\ell_1}:|\mathcal{S}_C|\geq (r-\ell_1)^2\right\}.
    \end{aligned}
    \end{equation*}
    Note that by the way we obtain $G_k$, we have $ U(G_k)\cup W= V(G_k)$.
    The following claim has been proved in \cite{helliar2024generalized}.

    \begin{claim}[\cite{helliar2024generalized}, Claims 5.2-5.4]\label{claim: sunflower}
    Assume $\mathscr{H}$ is $r$-uniform and $\{0,\ell\}$-intersecting where $0<\ell <r$.
    Let $U(\mathscr{H}) = \{u \in V(\mathscr{H}): d_{\mathscr{H}}(u)  \ge r^2\}$ and $\mathcal{C}(\mathscr{H}) = \{C \in \binom{V(\mathscr{H})}{\ell}: |\mathcal{S}_C| \ge r^2\}$. Then we have

    \noindent(a) $U(\mathscr{H})=\bigcup_{C\in \mathcal{C}(\mathscr{H})}C$,

    \noindent(b)  $C\cap C'=\emptyset$ for any distinct $C,C'\in \mathcal{C}(\mathscr{H}),$ and

    \noindent(c)  for every $A\in \mathscr{H}$ and every $C\in \mathcal{C}(\mathscr{H})$, we have either $C \subseteq A$ or $A\cap C=\emptyset$.
    \end{claim}

    Note that $\mathscr{H}_{G_k,F,W}^{r-\ell_1}$ is $(r-\ell_1)$-uniform and $\{0, \ell_2-\ell_1\}$-intersecting.
    Directly follows from the definitions, we have $U(G_k) = U(\mathscr{H}^{r-\ell_1}_{G_k,F,W})$ and $\mathcal{C}(G_k) = \mathcal{C}(\mathscr{H}^{r-\ell_1}_{G_k,F,W})$.
    By Claim \ref{claim: sunflower}, we know that $\mathcal{C}(G_k)$ forms a partition of $U(G_k)$ into $(\ell_2-\ell_1)$-subsets and thus $(\ell_2-\ell_1)\mid (r-\ell_1)$.

    It is worthy mentioning that if $\ell_2-\ell_1\nmid r-\ell_1$, $U(G_k) =\emptyset$ must hold.
    By (\ref{eq-1}), we have that $|\mathscr{H}_{G,F}^{r}| = O(n)$ which implies that $\Psi_r(n,F,\{\ell_1,\ell_2\})=O(n)$ in this case. 

    In the following, we assume $U(G_k)\neq \emptyset$ and prove the lemma by contradiction. Suppose $|\mathscr{H}^{r-\ell_1}_{G_k,F,W}|\geq cn^2$ for some constant $c$ when $n$ is sufficiently large.
    Assume $V(F) = \{v_1,v_2,\ldots,v_r\}$ and $W=\{w_1,w_2,\ldots,w_{\ell_1}\}$.
    By the way we define $W$, for any induced copy of $F$ in $G_k$ with vertex set $A$, we have $W\subseteq A$. Since $G_k[A]\cong F$, there exists a bijection $\phi_A$ from $V(F)$ to $A$ such that $uv \in E(F)$ if and only if $\phi_A(u)\phi_A(v) \in E(G_k[A])$.
    Given an $(r-\ell_1)$-partition of $V(G_k)\setminus W = V_{\ell_1+1} \cup V_{\ell_1+2} \cup \cdots \cup V_r$ and a permutation $\sigma$ over $[r]$, we say an induced copy of $F$  with vertex set $A \subseteq V(G_k)$ is \textbf{good}~(also say $A$ is good and denote $G_k[A]=A$ for short) if $\phi_A(v_{\sigma (i)}) = w_i, 1 \le i \le \ell_1$ and $\phi_A(v_{\sigma (i)}) \in V_{i}, \ell_1 < i \le r$.
    Now we uniformly randomly pick a permutation $\sigma'$ over $[r]$ and then uniformly random select an $(r-\ell_1)$-partition $V_{\ell_1+1}' \cup V_{\ell_1+2}' \cup \cdots \cup V_r' = V(G_k)\setminus W$. Given an induced copy of $F$  with vertex set $A \subseteq V(G_k)$, the probability that $A$ is good is a positive constant depending on $\ell_1$ and $r$, denoted by $p_1(\ell_1,r)$.
    Then there exist a permutation $\sigma$ and a partition $V(G_k)\setminus W = V_{\ell_1+1} \cup V_{\ell_1+2} \cup \cdots \cup V_r$ such that the number of good induced copies of $F$ is at least $cp_1(\ell_1,r)n^2$ by $|\mathscr{H}^{r-\ell_1}_{G_k,F,W}|\geq cn^2$. We denote the set consisted of the good induced copies of $F$ corresponding to $\sigma$ and the partition $V(G_k)\setminus W = V_{\ell_1+1} \cup V_{\ell_1+2} \cup \cdots \cup V_r$ by $\mathcal{F}_{good}$. Then $| \mathcal{F}_{good}|\ge cp_1(\ell_1,r)n^2$.
    In the following, we fix the permutation $\sigma$ and the partition $V(G_k)\setminus W = V_{\ell_1+1} \cup V_{\ell_1+2} \cup \cdots \cup V_r$.

    Let $F_0 = \{v_{\sigma(i)}\}_{1 \le i \le \ell_1}$.
    Now suppose that we have a partition of $V(F) = F_0 \cup F_1 \cup \cdots \cup F_{(r-\ell_1)/(\ell_2-\ell_1)}$ satisfying $|F_i| = \ell_2-\ell_1, 1 \le i \le (r-\ell_1)/(\ell_2-\ell_1)$.
    For $A \in \mathcal{F}_{good}$,  the bijection $\phi_A$ satisfies that $\phi_A(v_{\sigma(i)}) = w_i, 1 \le i \le \ell_1$.
    It shows that the image of $F_0$ under $\phi_A$ is $W$.
    Define $\phi_A(F_i) = \{\phi_A(u)\}_{u \in F_i}, 1 \le i \le (r-\ell_1)/(\ell_2-\ell_1)$.
    We say $A$ is \textbf{great} if $\phi_A(F_i) \in \mathcal{C}(G_k)$ for any $1 \le i \le (r-\ell_1)/(\ell_2-\ell_1)$.

    Now we uniformly randomly pick a partition of $V(F)\setminus F_0 = F_1' \cup \cdots \cup F_{(r-\ell_1)/(\ell_2-\ell_1)}'$ that satisfies $|F'_i| = \ell_2-\ell_1, 1 \le i \le (r-\ell_1)/(\ell_2-\ell_1)$. Then for every $A \in \mathcal{F}_{good}$, the probability that $A$ is great is a positive constant depending on $\ell_1,\ell_2,r$, denoted by $p_2(\ell_1,\ell_2,r)$.
    Then there exists a partition of $V(F)\setminus F_0 = F_1 \cup \cdots \cup F_{(r-\ell_1)/(\ell_2-\ell_1)}$ with $|F_i| = \ell_2-\ell_1, 1 \le i \le (r-\ell_1)/(\ell_2-\ell_1)$ such that the number of great induced copies of $F$ is at least $cp_1(\ell_1,r)p_2(\ell_1,\ell_2,r)n^2$ by $| \mathcal{F}_{good}|\ge cp_1(\ell_1,r)n^2$. We denote the set consisted of the great induced copies of $F$  corresponding to the  partition  $V(F)\setminus F_0 = F_1 \cup \cdots \cup F_{(r-\ell_1)/(\ell_2-\ell_1)}$ by $\mathcal{F}_{great}$.
    Then $|\mathcal{F}_{great}|\ge cp_1(\ell_1,r)p_2(\ell_1,\ell_2,r)n^2$.
    In the following discussion, we fix the partition of $F$ as well.

    We define an auxiliary graph $G^\star$ with vertex set $V(G^\star)=\mathcal{C}(G_k)$ and edge set
    $$E(G^\star)=\big\{ \{C,C'\}:~ C,C'\in {\mathcal{C}(G_k),C\cup C'\cup W\subseteq F \text{~ for some $F\in \mathcal{F}_{great}$}} \big\}.$$
    Then we define an auxiliary hypergraph $\mathscr{H}^\star$ with vertex set $V(\mathscr{H}^\star) = \mathcal{C}(G_k)$ and edge set
    \[
    E(\mathscr{H}^\star) = \left\{ B \in \binom{\mathcal{C}(G_k)}{(r-\ell_1)/(\ell_2-\ell_1)}:~  \left(\bigcup_{C \in B}C \right)\cup W \in  \mathcal{F}_{great} \right\}.
    \]
    Note that each $F \in \mathcal{F}_{great}$ corresponds to a hyperedge in $\mathscr{H}^\star$.
    Then we have that
    \begin{equation}\label{eq}
   |\mathscr{H}^\star| \ge cp_1(\ell_1,r)p_2(\ell_1,\ell_2,r)n^2.
\end{equation}

Recalling that when $F$ is a clique, the problem is solved.
The following claim is the key of the proof which connects the general $F$ to the clique.


    \begin{claim}\label{claim: K_s respond to hyperedge}
        If $A \subseteq V(G^\star)$ and $|A| = (r-\ell_1)/(\ell_2-\ell_1)$, then $A$ is a clique in $G^\star$ if and only if $A$ is a hyperedge in $\mathscr{H}^{\star}$.
    \end{claim}

 \noindent
    \textbf{Proof of Claim~\ref{claim: K_s respond to hyperedge}.}
The sufficiency holds  by the definition of $G^\star$ and $\mathscr{H}^\star$. So we just need to proof the necessity.

        Recalling that we have fixed $\sigma$, the partition of $V(G_k)\setminus W = V_{\ell_1+1}\cup \cdots\cup V_r$ and the partition of $V(F)\setminus F_0 = F_1\cup \cdots\cup F_{(r-\ell_1)/(\ell_2-\ell_1)}$ with each $|F_i| = \ell_2-\ell_1, 1 \le i \le (r-\ell_1)/(\ell_2-\ell_1)$.

        For a vertex $u \in V(G_k)\setminus W$, let $\psi(u)$ be the index $j$ such that $u \in V_j$.
        For a set $C\in \mathcal{C}(G_k)$, write $\psi(C) = \{\psi(u)\}_{u \in C}$.
        For any $A  \in \mathcal{F}_{great}$ and any vertex $u \in A \setminus W$, we have $|(A \setminus W) \cap V_{\psi(u)}|=1$ since $A$ is good.
        Note that $u \in V_{\psi(u)}$.
        Then for any distinct $u_1,u_2\in A \setminus W$, we have $\psi(u_1)\neq \psi(u_2)$.
        Hence, if $C\in \mathcal{C}(G_k)$ is a subset of $A$, then $|\psi(C)|=\ell_2-\ell_1$.
         This means that every vertex of $C$ belongs to different parts among $V_{\ell_1+1}\cup\cdots \cup V_r$.
        If two elements $C_1,C_2\in \mathcal{C}(G_k)$ are adjacent in $G^\star$, then there exists $A_{C_1,C_2} \in \mathcal{F}_{great}$ such that $C_1\cup C_2\cup W \subseteq A_{C_1,C_2}$.
        By the same argument, we can derive $\psi(C_1) \cap  \psi(C_2) = \emptyset$.

       Assume $C_1,\ldots,C_{(r-\ell_1)/(\ell_2-\ell_1)}\in V(G^{\star})$ form a clique in $G^{\star}$. Let $B=\bigcup_{i=1}^{(r-\ell_1)/(\ell_2-\ell_1)}C_i$.
        Then $\{\psi(C_i)\}_{1 \le i \le (r-\ell_1)/(\ell_2-\ell_1)}$ is a partition of $\{\ell_1+1,\ldots,r\}$.
        This means that $ \left|B\cap V_{j}\right|=1$ for every $j=\ell_1+1,\ldots,r$.
        Once we know that a vertex $u \in B$ satisfies that $\psi(u) = j \in \{\ell_1+1,\ldots,r\}$, then we can conclude that $\{u\} = B\cap V_{j}$.
        This will be an important fact in our proof.

        Now we prove $G_k\left[B\cup W\right]\cong F$.
        We construct a bijection $\phi'$ from $V(F)$ to $B\cup W$.
        Let $\phi'(v_{\sigma(j)}) = w_j, 1 \le j \le \ell_1$ and $\phi'(v_{\sigma(j)}) = z_j$ where $\{z_j \} = B \cap V_j$ for $\ell_1 < j \le r$.
        The mapping is well-defined by our previous arguments.
        Now we only need to prove $v_{\sigma(i)}v_{\sigma(j)} \in E(F)$ if and only if $z_iz_j\in E\left(G_k\left[B\cup W\right]\right)$ for any $i \neq j$.
        Without loss of generality, assume $z_i \in C_1\cup W$ and $z_j \in C_1\cup C_2\cup W$.

        Since $A_{C_1,C_2} $ is an induced copy of $F$, there is $\phi_{A_{C_1,C_2}}$ such that $\phi_{A_{C_1,C_2}}(v_{\sigma(i)})\phi_{A_{C_1,C_2}}(v_{\sigma(i)}) \in E(G_k)$ if and only if $v_{\sigma(i)}v_{\sigma(j)} \in E(F)$.
        If $i \le \ell_1$, then $\phi_{A_{C_1,C_2}}(v_{\sigma(i)}) = w_i = z_i$ since $A_{C_1,C_2}$ is good.
        If $ \ell_1 < i \le r$, we have $\phi_{A_{C_1,C_2}}(v_{\sigma(i)}) \in V_i$ and note that $\psi(z_i) = i$ which means $z_i$ is the only vertex in $A_{C_1,C_2} \cap V_i$. Thus $\phi_{A_{C_1,C_2}}(v_{\sigma(i)}) = z_i$.
        As a result, we derive that $\phi_{A_{C_1,C_2}}(v_{\sigma(i)}) = z_i$ and similarly we can derive $\phi_{A_{C_1,C_2}}(v_{\sigma(j)}) = z_j$.
        Thus $v_{\sigma(i)}v_{\sigma(j)} \in E(F)$ if and only if $z_iz_j\in E(G_k[B\cup W])$ for any $i \neq j$. So $G_k[B  \cup W]\cong F$.

        By the definition, we can easily verify that $B\cup W\in \mathcal{F}_{good}$.
        We then prove $B\cup W\in \mathcal{F}_{great}$. For convenience, we write $p = (r-\ell_1)/(\ell_2-\ell_1)$.
        For any $F_i$, $1 \le i \le p$, let $\eta(F_i) = \{ j: v_{\sigma(j)} \in F_i\}$. It is easy to verify that $\eta(F_i) = \{\psi(u)\}_{u \in \phi'(F_i)}$.
        Since $B\cup W$ is good, $\{\eta(F_i)\}_{1 \le i \le p}$ is a partition of $\{\ell_1+1,\ldots, r\}$.

        Since each $C_i$ is contained in some $A_{C_i} \in \mathcal{F}_{great}$,  there is  $F_{C_i} \in \{F_1,F_2,\ldots,F_{p}\}$ such that $\{\phi_{A_{C_i}}(v)\}_{v\in F_{C_i}} = C_i$.
        Then we have that
        \[
        \psi(C_i) = \{\psi(\phi_{A_{C_i}}(v_{\sigma(j)})): v_{\sigma(j)} \in F_{C_i}\} = \{j: v_{\sigma(j)} \in F_{C_i}\} = \eta(F_{C_i}).
        \]
        That is, $\psi(C_i) \in \{\eta(F_j)\}_{1 \le j \le p}$ for every $1 \le i \le p$. Since $\{\psi(C_i)\}_{1 \le i \le p}$ is also a partition of $\{\ell_1+1,\ldots, r\}$,
        we have that $\{\psi(C_i)\}_{1 \le i \le p} = \{\eta(F_j)\}_{1 \le j \le p}$.
        Now we claim that for any $i$, $\phi'(F_i) \in \{C_1,\ldots,C_{p}\}$; otherwise, $\eta(F_i) = \{\psi(u)\}_{u \in \phi'(F_i)} \notin \{\psi(C_i)\}_{1 \le j \le p}$, a contradiction.
        Therefore, $\{C_1,C_2,\ldots,C_{(r-\ell_1)/(\ell_2-\ell_1)}\}$ is a hyperedge in $\mathscr{H}^\star$.    \hfill $\blacksquare$ \par

    \begin{claim}\label{claim: 0,1 intersecting}
        $G^\star$ is $(K_{(r-\ell_1)/(\ell_2-\ell_1)},\{0,1\})$-intersecting.
    \end{claim}

    \noindent
    \textbf{Proof of Claim~\ref{claim: 0,1 intersecting}.}
    By Claim~\ref{claim: K_s respond to hyperedge}, we have that $\mathscr{H}^\star = \mathscr{H}_{G,K_{(r-\ell_1)/(\ell_2-\ell_1)}}^{(r-\ell_1)/(\ell_2-\ell_1)}$.
    Each hyperedge in $\mathscr{H}^\star$~(together with $W$) corresponds to an induced copy of $F$ in $G_k$.
    Since $G_k$ is $(F,L)$-intersecting, by the definition, we have that $\mathscr{H}^\star$ is $\{0,1\}$-intersecting.
    As a whole, we conclude that $G^\star$ is $(K_{(r-\ell_1)/(\ell_2-\ell_1)},\{0,1\})$-intersecting.
    \hfill $\blacksquare$ \par

    Now let us finish the proof of the lemma.
    Note that $|V(G^\star)| = |\mathcal{C}(G_k)| \le \frac{|V(G_k)| - \ell_1}{\ell_2-\ell_1} \le n$ from Claim~\ref{claim: sunflower}. By Claim \ref{claim: K_s respond to hyperedge}, $
    |\mathscr{H}^\star|=N_{ind}(G^\star, K_{(r-\ell_1)/(\ell_2-\ell_1)})$.
    By Claim \ref{claim: 0,1 intersecting} and (\ref{eq: general of (6,3)}), we have
    $
    |\mathscr{H}^\star| = o(n^2),
    $ a contradiction with (\ref{eq}).
             Here we finish the proof of Lemma~\ref{lem: L=l1,l2}.
   \end{proof}

Lemma~\ref{lem: L=l1,l2} proves the base case of Theorem~\ref{thm: the main theorem} when $|L|=s=2$.
 For more general $L$ and $r$ with similar restrictions, we will still employ the induction on both $r$ and $|L|$ as in \cite{helliar2024generalized} and \cite{zhao2025counting}. We need Lemma~\ref{lem: l1,l2,l3} from~\cite{zhao2025counting} and a consequence of F{\"u}redi's fundamental structure theorem from~\cite{furedi1983finite}.

\begin{lemma}[Zhao-Zhang \cite{zhao2025counting}]\label{lem: l1,l2,l3}
    Let $0\le \ell_1<\ell_2<\ell_3<r$ with $r-\ell_3=\ell_3-\ell_2\neq \ell_2 - \ell_1$. Then any $n$-vertex $\{\ell_1,\ell_2,\ell_3\}$-intersecting $r$-graph $\mathscr{H}$ satisfies
    $$|\mathscr{H}|\leq \binom{n}{2}=O(n^2).$$
    That is, $\Phi_r(n, \{\ell_1,\ell_2,\ell_3\}) = O(n^2)$\footnote{In~\cite{zhao2025counting}, they only prove the case when $\ell_1=0$. Here when $\ell_1 >0$, it can be easily derived from Theorem~\ref{thm: deza-E-F}.}.
\end{lemma}

 \begin{theorem}[F{\"u}redi~\cite{furedi1983finite}]\label{thm: furedi}
     For $r\geq 2$ and $L\subseteq [0,r-1]$, there exists a positive constant $c=c(r)$ such that every $L$-intersecting $r$-graph $\mathcal{F}$ contains a subhypergraph $\mathcal{F}^\star\subseteq\mathcal{F}$ with $|\mathcal{F}^\star|\geq c|\mathcal{F}|$ satisfying all of the following properties.

     (i) The families $\mathcal{I}(H):=\{H\cap H':~H'\in \mathcal{F}^\star\}$ are isomorphic for all $H\in \mathcal{F}^\star$.

     (ii) For any $A\in \mathcal{I}(H)$ with $H\in \mathcal{F}^\star$, $A$ is the core of an $(r+1)$-sunflower in $\mathcal{F}^\star$.

     (iii) For any $H\in \mathcal{F}^\star$, $\mathcal{I}(H)$ is closed under intersection, i.e., if $A_1,A_2\in \mathcal{I}(H)$ then $A_1\cap A_2\in \mathcal{I}(H)$.

     (iv) For any $A\in \mathcal{I}(H)$ with $H\in \mathcal{F}^\star$, $|A|\in L$.

     (v) For any different $A,A'\in\bigcup_{H\in \mathcal{F}^\star}\mathcal{I}(H)$, we have $|A\cap A'|\in L$.
 \end{theorem}
 The following lemma gives the reduction of $\Psi_r(n,F,L)$.
\begin{lemma}\label{lem: reduction lemma}
    For $L=\{\ell_1,\dots,\ell_s\}\subseteq[0,r-1]$ with $0\leq \ell_1<\dots<\ell_s<r$, if $\Psi_r(n,F,L)\geq C(r,L)n^{s-1}$ where $C(r,L)$ is the constant in Theorem \ref{thm: deza-E-F}, we have
    \begin{equation*}
    \begin{aligned}
        &\Psi_r(n,F,L)  = \Psi_r(n,F,L,\ell_1) \\
        \leq& \frac{1}{c(r)}\min_{1\le i\le s}\max\{\Phi_r(n,L\setminus\{\ell_i\}),\Phi_{\ell_i}(n,\{\ell_1,\dots,\ell_{i-1}\})  \Psi_r(n,F,\{\ell_i,\dots,\ell_s\},\ell_i)\},
    \end{aligned}
    \end{equation*}
    where $c(r)$ is the constant stated in Theorem \ref{thm: furedi}.
\end{lemma}
\begin{proof}
    It suffices to prove that for any $i \in [s]$,

    \begin{equation*}
    \begin{aligned}
        &\Psi_r(n,F,L)  = \Psi_r(n,F,L,\ell_1) \\
        \leq& \frac{1}{c(r)}\max\{\Phi_r(n,L\setminus\{\ell_i\}),\Phi_{\ell_i}(n,\{\ell_1,\dots,\ell_{i-1}\})  \Psi_r(n,F,\{\ell_i,\dots,\ell_s\},\ell_i)\}.
    \end{aligned}
    \end{equation*}
    In the following, let $i$ be a fixed integer in $[s]$.
    Let $G$ be an $(F,L)$-intersecting graph with $n$ vertices such that $N_{ind}(G,F)=\Psi_r(r,F,L)$. Since $\Psi_r(n,F,L)\geq C(r,L)n^{s-1}$, by Theorem \ref{thm: deza-E-F}, there is $W\subseteq \bigcap_{\text{$A\in \mathscr{H}_{G,F}^{r}$}}A$ with $|W|=\ell_1$.
    Therefore, we have $\Psi_r(n,F,L,\ell_1)=\Psi_r(n,F,L)$.
    Let $\mathcal{F}=\mathscr{H}^r_{G,F}$. Then we have $|\mathcal{F}|=\Psi_r(r,F,L,\ell_1)$. Applying Theorem \ref{thm: furedi} to $\mathcal{F}$, we can find a subhypergraph $\mathcal{F}^\star\subseteq \mathcal{F}$ satisfying all of the properties in Theorem \ref{thm: furedi}.
    If there is no $A\in \bigcup_{H\in \mathcal{F}^\star}\mathcal{I}(H)$ such that $|A|=\ell_i$, then by the definition of $\mathcal{I}(H)$, we have that $\mathcal{F}^\star$ is $L\setminus\{\ell_i\}$-intersecting. So $|\mathcal{F}|\leq c(r)^{-1}|\mathcal{F}^\star|\leq c(r)^{-1}\Phi_r(n,L\setminus\{\ell_i\})$ and we are done.

    Thus, by Theorem \ref{thm: furedi} ($i$), we may assume that for every $H\in \mathcal{F}^\star$ there exists some $A\in \mathcal{I}(H)$ of cardinality $\ell_i$.
    Let $\mathcal{A}^i_H:=\{A\in \mathcal{I}(H):~|A|=\ell_i\}$. Then $\mathcal{A}^i_H\not=\emptyset$  for each $H\in \mathcal{F}^\star$. By Theorem \ref{thm: furedi} ($v$), it is clear that $\bigcup_{H\in \mathcal{F}^\star}\mathcal{A}^i_H$ is an $\{\ell_1,\dots,\ell_{i-1}\}$-intersecting $\ell_i$-graph and hence
    \begin{equation}
        \left|\bigcup_{H\in \mathcal{F}^\star}\mathcal{A}^i_H \right|\leq \Phi_{\ell_i}(n,\{\ell_1,\dots,\ell_{i-1}\}).
    \end{equation}
    For $A\in \bigcup_{H\in \mathcal{F}^\star}\mathcal{A}^i_H$, let $\mathcal{F}^\star(A)$ be the subhypergraph of $\mathcal{F}^\star$ with all edges contained $A$.
    Then $G[\bigcup_{F\in\mathcal{F}^\star(A)}V(F)]$ is $(F,\{\ell_i,\dots,\ell_s\})$-intersecting.
    Since each hyperedge in $\mathcal{F}^\star(A)$ is an induced copy of $F$ in $G$, we use the notation $F$ here to represent the hyperedges.
    So $|\mathcal{F}^\star(A)|\leq \Psi_r(n,F,\{\ell_i,\dots,\ell_s\},|A|)=\Psi_r(n,F,\{\ell_i,\dots,\ell_s\},\ell_i)$.

    Since for any $H\in \mathcal{F}^\star$, $\mathcal{A}^i_H$ contains at least one $\ell_i$-set, we have
    \begin{equation*}
        \begin{aligned}
            \left|\mathcal{F}^\star \right|\leq& \left| \bigcup_{A\in\bigcup_{H\in \mathcal{F}^\star}\mathcal{A}^i_H}\mathcal{F}^\star(A) \right|\leq \left| \bigcup_{H\in \mathcal{F^\star}}\mathcal{A}^i_H \right|\cdot\Psi_r(n,F,\{\ell_i,\dots,\ell_s\},\ell_i)\\
            \leq & \Phi_{\ell_i}(n,\{\ell_1,\dots,\ell_{i-1}\}) \cdot \Psi_r(n,F,\{\ell_i,\dots,\ell_s\},\ell_i).
        \end{aligned}
    \end{equation*}
    Then we finish the prove by $|\mathcal{F}|\leq c(r)^{-1}|\mathcal{F}^\star|$.
    \end{proof}


    \noindent
    {\bf Proof of Theorem \ref{thm: the main theorem}:}

    We proof the theorem by induction on $|L|$.
    When $|L|=2$, it is proved in Lemma~\ref{lem: L=l1,l2}.
    When $|L|=s\geq 3$, we assume that the theorem holds for smaller $|L|$.

    Since $\ell_1,\dots,\ell_s,r$ do not form an arithmetic progression,  either $r-\ell_s\neq \ell_s-\ell_{s-1}$ or there exists $i\in [s-2]$ such that $r-\ell_s=\ell_s-\ell_{s-1}=\dots=\ell_{i+2}-\ell_{i+1}\neq \ell_{i+1}-\ell_i$. Let $G$ be a $(F,L)$-intersecting graph with maximum number of induced copies of $F$, that is, $N_{ind}(G,F)=\Psi_r(n,F,L)$. To prove the upper bound, we may assume $\Psi_r(n,F,L)\geq C(r,L)n^{s-1}$ where $C(r,L)$ is the constant in Theorem \ref{thm: deza-E-F}. Otherwise, we are done.

    \begin{mycase}

        \case  $r-\ell_s\neq \ell_{s}-\ell_{s-1}$. In this case,
        by Lemma \ref{lem: reduction lemma}, we have
        \begin{equation*}
            \begin{aligned}
                \Psi_r(n,F,L)\leq &\frac{1}{c(r)}\max\big\{\Phi_r(n,F,L\setminus\{\ell_{s-1}\}),\Phi_{\ell_{s-1}}(n,\{\ell_1,\dots,\ell_{s-2}\})\Psi_r(n,F,\{\ell_{s-1},\ell_s\},\ell_{s-1})\big\}\\
                \leq &c(r)^{-1}\big\{n^{s-1},n^{s-2}\Psi_r(n,F,\{\ell_{s-1},\ell_s\},\ell_{s-1})\big\}.
            \end{aligned}
        \end{equation*}
        It follows from Lemma \ref{lem: L=l1,l2} that $\Psi_r(n,F,\{\ell_{s-1},\ell_s\},\ell_{s-1})=o(n^2)$, and if $\ell_s-\ell_{s-1}\nmid r-\ell_{s-1}$, we have $\Psi_r(n,F,\{\ell_{s-1},\ell_s\},\ell_{s-1})=O(n)$.
    Thus, $\Psi_r(n,F,L)=o(n^s)$ and if $r-\ell_{s-1}\nmid \ell_2-\ell_1$, we have $\Psi_r(n,F,L)=O(n^{s-1})$ .

    \case There exists $i\in [s-2]$ such that $r-\ell_s=\ell_s-\ell_{s-1}=\dots=\ell_{i+2}-\ell_{i+1}\neq \ell_{i+1}-\ell_i$.
   By Lemma \ref{lem: reduction lemma}, we have
    \begin{equation*}
        \begin{aligned}
            \Psi_r(n,F,L)\leq& c(r)^{-1}\max\{\Phi_r(n,L\setminus\{\ell_i\}),\Phi_{\ell_i}(n,\{\ell_1,\dots,\ell_{i-1}\}) \Psi_r(n,F,\{\ell_i,\dots,\ell_s\},\ell_i)\}\\
            \leq &
            c(r)^{-1}\max\{n^{s-1},n^{i-1}\Psi_r(n,F,\{\ell_i,\dots,\ell_s\},\ell_i)\}.
        \end{aligned}
    \end{equation*}

    If $i>1$, then by induction, we have that $\Psi_r(n,F,\{\ell_i,\dots,\ell_s\},\ell_i)=O(n^{s-i})$ and we are done.
    If $i = 1$ and $s=3$, the result holds by Lemma~\ref{lem: l1,l2,l3}.
    If $i=1$ and $s \ge 4$, by Lemma \ref{lem: l1,l2,l3}, we have that
    \begin{equation}\label{eq: l1,l2,l3}
        \Phi_{\ell_{4}}(n,\{\ell_{1}, \ell_{2},\ell_{3}\}) \leq \binom{n}{2} = O(n^2).
    \end{equation}
    And by Lemma \ref{lem: reduction lemma}, we have
    \begin{equation*}
        \begin{aligned}
            &\Psi_r(n,F,L)\\
            \leq & c(r)^{-1}\max\{\Phi_r(n,L\setminus\{\ell_{4}\}), \Phi_{\ell_{4}}(n,\{\ell_{1}, \ell_{2},\ell_{3}\}) \cdot
            \Psi_r(n,F,\{\ell_{4},\dots,\ell_s\},\ell_{4})\}\\
            \leq & c(r)^{-1}\max\{n^{s-1}, \Phi_{\ell_{4}}(n,\{\ell_{1}, \ell_{2},\ell_{3}\})\cdot n^{s-3}\}\leq O(n^{s-1}),
        \end{aligned}
    \end{equation*} where the last inequality holds by combining (\ref{eq: l1,l2,l3}).
        \end{mycase}
        \hfill$\square$\par


\subsection{Proof of Theorem~\ref{thm: the main theorem 3}}

The proof of Theorem~\ref{thm: the main theorem 3} is almost the same as the proof of Theorem~\ref{thm: the main theorem}.
We only need to extend the following lemmas.

\begin{lemma}\label{lem: L=l1,l2-span}
    Let $r\geq 3$, $0\leq \ell_1<\ell_2$ with $r-\ell_2\neq \ell_2-\ell_1$, and $F$  a graph with $r$ vertices. Then $\Psi_r(n,F,\{\ell_1,\ell_2\})=o(n^2)$. Moreover, if $\ell_2-\ell_1\nmid r-\ell_1$, $\Psi_r(n,span(F),\{\ell_1,\ell_2\})=O(n)$.
\end{lemma}

\begin{proof}
    Lemma~\ref{lem: L=l1,l2-span} is the extension of Lemma~\ref{lem: L=l1,l2}. In the proof of Lemma~\ref{lem: L=l1,l2}, we characterize the vertex set $A \subseteq V(G)$ which forms an induced copy of $F$ by a bijection $\phi_A: V(F) \rightarrow A$ such that $uv\in E(F)$ if and only if $\phi_A(u)\phi_A(v) \in E(G[A])$. Now we consider the copies of $F$~(not necessarily induced). If $G[A]$ has a copy of $F$, then we can characterize it also by a bijection $\phi_A': V(F) \rightarrow A$ such that if $uv \in E(F)$ then $\phi'_A(u)\phi'_A(v) \in E(G[A])$. In the rest parts, we only need to repeat the proof. The details are omitted here.
\end{proof}


\begin{lemma}\label{lem: reduction lemma-span}
For $L=\{\ell_1,\dots,\ell_s\}\subseteq[0,r-1]$ with $0\leq \ell_1<\dots<\ell_s<r$, if $\Psi_r(n,span(F),L)\geq C(r,L)n^{s-1}$ where $C(r,L)$ is the constant in Theorem \ref{thm: deza-E-F}, we have
\begin{equation*}
    \begin{aligned}
       &\Psi_r(n,span(F),L,\ell_1) = \Psi_r(n,span(F),L)\\
       \leq & \frac{N(K_r,F)}{c(r)}\min_{1\le i\le s}\max\{\Phi_r(n,L\setminus\{\ell_i\}),~\Phi_{\ell_i}(n,\{\ell_1,\dots,\ell_i\}) \cdot \Psi_r(n,span(F),\{\ell_i,\dots,\ell_s\},\ell_i)\},
    \end{aligned}
\end{equation*}
where $c(r)$ is the constant stated in Theorem \ref{thm: furedi} and $N(K_r,F)$ is the number of the copies of $F$ in $K_r$.
\end{lemma}

Lemma~\ref{lem: reduction lemma-span} is the extension of Lemma~\ref{lem: reduction lemma}.
With  Lemmas \ref{lem: L=l1,l2-span} and \ref{lem: reduction lemma-span}, the proof of Theorem \ref{thm: the main theorem 3} is similar to Theorem \ref{thm: the main theorem}. We omit the details here.


\section{Preliminary}\label{sec: preliminary}

\subsection{ The entropy method}

We first introduce basic properties of the entropy method.
For a complete discussion and proofs for properties of the entropy method, readers may refer to Section 14.6 in~\cite{alon2016probabilistic}.

\begin{definition}\label{def: entropy}
	For any discrete random variable $X$ with finite support $S$, the \textbf{entropy} of $X$ is defined as
	\begin{equation*}
			H(X) = -\sum_{x \in S} p_X(x) \log p_X(x),
	\end{equation*}
	where $p_X(x)$ denotes  $\mathbb{P}(X=x)$.
\end{definition}

The entropy $H(X_1, \ldots, X_n)$ of several discrete random variables $X_1, \ldots, X_n$ is defined as the entropy of the random variable $X=(X_1, \ldots, X_n)$.

\begin{definition}[Conditional Entropy]\label{def: conditional entropy}
	Let $X,Y$ be two random variables with finite supports $S,T$, respectively.
	Denote by $p_{X,Y}(x,y)$ the probability $\mathbb{P}(X=x,Y=y)$.
	For each $y \in T$, let $X|Y=y$ be the random variable $X$ conditioned on $Y=y$, and $H(X \mid Y=y)$ be its entropy.
	The conditional entropy of $X$ given $Y$ is defined as
	\begin{equation*}
		H(X \mid Y) = \sum_{y \in T} p_Y(y) H(X\mid Y=y) = -\sum_{x \in S, y \in T} p_{X,Y}(x,y) \log \left(\frac{p_{X,Y}(x,y)}{p_Y(y)}\right).
	\end{equation*}
\end{definition}

\begin{proposition}\label{prop: entropy support bound}
	For any random variable $X$ with finite support $S$, we have $H(X) \le \log |S|$.
	The equality holds when $X$ is uniformly distributed over $S$.
\end{proposition}

\begin{proposition}[Chain Rule]\label{prop: chain rule}
	For any two random variables $X$ and $Y$ with finite supports, we have
	\begin{equation*}
		H(X,Y) = H(X) + H(Y\mid X) = H(Y) + H(X\mid Y).
	\end{equation*}
	More generally, for any $n$ random variables $X_1, \ldots, X_n$ with finite supports, we have
	\begin{equation*}
		H(X_1, \ldots, X_n) = H(X_1) + H(X_2\mid X_1) + \cdots + H(X_{n} \mid X_{1}, \ldots, X_{n-1}).
	\end{equation*}
\end{proposition}

\begin{proposition}[Drop Condition]\label{prop: drop condition}
	For any two random variables $X$ and $Y$ with finite supports, we have
	\[H(X \mid Y) \le H(X).\]
\end{proposition}

We will use the entropy method to count the induced cycles in some graph $G$.
The method is essentially a re-translated version of the counting method in~\cite{norin2019bound} used to count induced cycles.

\subsection{Atom}\label{sec: atom}
Here we will introduce some basic definitions and lemmas that will help us deal with the situation when $L$ and $r$ form an arithmetic progression and $F=C_r$ in the following proof of Theorem~\ref{thm: the main theorem 2}.

Now we assume $n$ is sufficiently large.
Assume $L = \{\ell_1,\ell_2,\ldots,\ell_s\}$ and $\ell_1,\ell_2,\ldots,\ell_s,r$ form an arithmetic progression with common difference $d$, that is, $\ell_{i+1}-\ell_i = d$ for all $i=1,2,\ldots,s$~(write $\ell_{s+1}=r$). Then $r=\ell_1+sd$.
We will first prove that $\Psi_r(n,C_r,L)\geq (1+o(1))\left(\frac{n-\ell_1}{r-\ell_1}\right)^s$ in Sections~\ref{sec: proof AP 1} and \ref{sec: proof AP 2-4}.
Let $G$ be the extremal $(C_r, L)$-intersecting graph with $n$ vertices maximizing $N_{ind}(G,C_r)$.
That is, $\Psi_r(n,C_r,L) = N_{ind}(G,C_r)$.
By the lower bound and Theorem~\ref{thm: deza-E-F}, there exists $W \subseteq \bigcap_{A \in \mathscr{H}_{G,C_r}^{r}}A$ such that $|W| = \ell_1$. 

Here, we define \textbf{atom} as in \cite{zhao2025counting}.
We say that a subset $S\subseteq V(G)\setminus W$ with $|S|\geq d$ is an \textbf{atom} if $S$ is inclusion maximal with the property that either $S\subseteq A$ or $S\cap A=\emptyset$ for any $A\in \mathscr{H}^{r-\ell_1}_{G,C_r,W}$.
Then the atoms are pairwise disjoint. Let $\mathcal{S}\subseteq \binom{V(G)}{d}$ be the set consisting of all atoms of size $d$ and let $X_1:=\bigcup_{S\in \mathcal{S}}S\subseteq V(G)$.
Since atoms are disjoint sets, we have that
\begin{equation}\label{eq: ub S}
    |\mathcal{S}| \le (|V(G)|-|W|)/d=(n-\ell_1)/d,
\end{equation}
which is important in our following proof.

Recalling that $X_1 = \bigcup_{S \in\mathcal{S}}S$, we can also view $\mathcal{S}$ as a partition of $X_1$ into sets of size $d$.

Let $H$ be a graph with vertex set $X \cup Y$ satisfying that $d \mid |X|$ and $|Y|=\ell_1$.
Let $\mathcal{S}'$ be a partition of $X$ such that each set has size $d$.
For an induced copy of $C_r$ in $H$, denoted by $C$, we say  $C$ is \textbf{standard} with $(X, Y, \mathcal{S}')$ if there are $S_1,\ldots,S_s\in \mathcal{S}'$ such that $V(C) = \left( \bigcup_{i=1}^{s}S_i \right) \cup Y$.
By the definition of atoms, all induced copies of $C_r$ in $G[X_1\cup W]$ are standard with $(X_1, W,\mathcal{S})$.
Let $\Upsilon(H, X, Y, \mathcal{S}', C_r)$ be the collection of induced copies of $C_r$ standard with $(X, Y, \mathcal{S}')$ in $H$.
Then we have
\begin{equation}\label{eq: Upsilon with scr H}
 N_{ind}(G[X_1\cup W], C_r)  = |\Upsilon(G[X_1\cup W], X_1, W, \mathcal{S}, C_r)|.
\end{equation}
The purpose is to relax the condition that all induced copies of $C_r$ must be $L$-intersecting.
Note that if all induced copies of $C_r$ in $H[X\cup Y]$ are standard with $(X, Y,\mathcal{S}')$, then $H$ is $(C_r, L)$-intersecting.
In other words, we only count those standard induced copies of $C_r$. We permit non-standard cycles to arise but do not count them.

\begin{lemma}\label{lem: e HW has edge, Upsilon}
    When $s\geq 2$ and $E(H[Y]) \neq \emptyset$, we have $|\Upsilon(H,X,Y,\mathcal{S}',C_r)|\leq (1+o(1))\left( \frac{|\mathcal{S'}|}{s} \right)^s$.
\end{lemma}

\begin{proof}
   Denote $Y=\{y_1,\dots,y_{\ell_1}\}$. Since $Y$ is contained in all standard induced copies of $C_r$ and $E(H[Y]) \neq \emptyset$, there exists a $y_i\in Y$, say $y_1$, such that  $d_{G[Y]}(y_1)=1$. Let $$\mathcal{S}'':=\{S\in \mathcal{S}':~\text{there exist $u\in S$ such that $uy_1\in E(H)$ }\}.$$
    We first define a new auxiliary graph $G_\Upsilon$ with  vertex set $\mathcal{S'}$ and  edge set
    $$E(G_\Upsilon)=\bigg\{\{S_1,S_2\}\in \binom{\mathcal{S'}}{2}:~ S_1\cup S_2\subseteq V(C_r) \text{~for some $C_r\in \Upsilon(H,X,Y,\mathcal{S}',C_r)$}\bigg\}.$$
    We will prove the lemma by induction on $s$.
    When $s=2$, $|E(G_\Upsilon)| = |\Upsilon(H,X,Y,\mathcal{S}',C_r)|$.
    We have $E(G_\Upsilon[\mathcal{S}''])=\emptyset$ since the degree of $y_1$ is $2$ in every induced $C_r$ and $d_{G[Y]}(y_1)=1$. Similarly, we have $E(G_\Upsilon[\mathcal{S}'\setminus\mathcal{S}''])=\emptyset$. Thus $G_\Upsilon$ is $K_3$-free and we are done by Tur\'an theorem.

    Let $s\geq 3$. Since $d_{G[Y]}(y_1)=1$, every $C_r\in \Upsilon(H,X,Y,\mathcal{S}',C_r)$ contains exactly one element in $\mathcal{S}''$. Thus we have
    $$|\Upsilon(H,X,Y,\mathcal{S}',C_r)|=\sum_{S\in \mathcal{S}''}|\Upsilon(H[X\setminus(\cup_{U\in \mathcal{S}''}U)\cup Y\cup S],X\setminus(\cup_{U\in \mathcal{S}''}U),Y\cup S,\mathcal{S}'\setminus\mathcal{S}'',C_r)|.$$
    Since $E(H[Y\cup S])\neq \emptyset$, by the induction hypothesis, we have
    $$|\Upsilon(H[X\setminus(\cup_{U\in \mathcal{S}''}U)\cup Y\cup S],X\setminus(\cup_{U\in \mathcal{S}''}U),Y\cup S,\mathcal{S}'\setminus\mathcal{S}'',C_r)|\leq (1+o(1))\left(\frac{|\mathcal{S'}|-|\mathcal{S}''|}{s-1} \right)^{s-1}.$$
    Then we have $$|\Upsilon(H,X,Y,\mathcal{S}',C_r)|\leq (1+o(1))|\mathcal{S}''|\left(\frac{|\mathcal{S}'|-|\mathcal{S''}|}{s-1}\right)^{s-1}\leq (1+o(1))\left( \frac{|\mathcal{S'}|}{s} \right)^s.$$
    The equality holds when $|\mathcal{S''}|=\frac{|\mathcal{S}'|}{s}$.
\end{proof}

By Lemma~\ref{lem: e HW has edge, Upsilon}, (\ref{eq: Upsilon with scr H}) and (\ref{eq: ub S}), we have the following corollary.

\begin{cor}\label{cor: eG[W] is independent}
    When $|L|=s\geq 2$ and $E(G[W])\neq\emptyset$, we have $N_{ind}(G[X_1\cup W],C_r)\leq (1+o(1))\left( \frac{|\mathcal{S}|}{s} \right)^s=(1+o(1))\left(\frac{n-\ell_1}{r-\ell_1}\right)^s$.
\end{cor}

Let $X_0:=V(G)\setminus X_1$. The following lemma shows that for every $x\in X_0$, the number of hyperedges in $\mathscr{H}^{r-\ell_1}_{G, C_r, W}$ contained $x$ is very small, hence, the number of induced copies of $C_r$ contained $W$ and $x$ is very small. The lemma has been proved by Frankl and Tokushige \cite{frankl2016uniform}.
\begin{lemma}[Frankl and Tokushige \cite{frankl2016uniform}]\label{lem: X_0 has small number of Cr}
    For any $x\in X_0$, we have $$|\{A:~x\in A\in \mathscr{H}^{r-\ell_1}_{G,C_r,W}\}|=O(n^{s-2}).$$
\end{lemma}
By Lemma \ref{lem: X_0 has small number of Cr}, the number of induced copies of $C_r$ contained $X_0$ is at most $O(n^{s-1})$. Hence we have
\begin{equation}\label{eq: only consider X1}
N_{ind}(G,C_r)\leq N_{ind}(G[X_1\cup W],C_r)+O(n^{s-1}).
\end{equation}

(\ref{eq: only consider X1}) is an important equation in the following proof. It indicates that we only need to consider the induced copies of $C_r$ contained in $X_1\cup W$.
By (\ref{eq: only consider X1}) and Corollary~\ref{cor: eG[W] is independent}, we have the following result.

\begin{lemma}\label{lem: W not independent set psi}
    Assume $L = \{\ell_1,\ell_2,\ldots,\ell_s\}$ ($s\ge 2$) and $\ell_1,\ell_2,\ldots,\ell_s,r$ form an arithmetic progression with common difference $d$.
    Let $G$ be the $(C_r,L)$-intersecting graph with $n$ vertices such that there is $W\subseteq \bigcap_{A \in \mathscr{H}^r_{G,C_r}}A  $ and $|W|=\ell_1$. If $E(G[W])\neq\emptyset$, then
    \[
       N_{ind}(G,C_r) \le (1+o(1))\left( \frac{n-\ell_1}{r-\ell_1} \right)^s.
    \]
    Moreover, the above inequality holds if $\ell_1 > sd$.
\end{lemma}
{\bf Remark.} If $\ell_1 > sd$, then $E(G[W])\neq\emptyset$ by $W\subseteq V(C_r)$ and $r=\ell_1+sd$, where $C_r$ is an induce cycle of $G$.

\section{Proof of Theorem \ref{thm: the main theorem 2}~(1)}\label{sec: proof AP 1}
In this section, $L=\{\ell_1,\ell_2,\dots,\ell_s\}\subseteq [1,r-1]$~($s > 1$) and $\ell_1,\ell_2,\ldots,\ell_s,r$ form an arithmetic progression with common difference $d=1$. That is, $\ell_{i+1}-\ell_i = d=1, i=1,2,\ldots,s$ and we make a deal that $\ell_{s+1}=r$. Then $r=\ell_1+s$.

For convenience, we rewrite the result here:
$$ (1)~\Psi_r(n,C_r,L)=\left\{
\begin{aligned}
&(1+o(1))\left(\frac{n-\ell_1}{r-\ell_1}\right)^s, \quad~\text{when $(\ell_1,s)\not\in \{(1,3),(2,2)\}$,} \\
&(2+o(1))\left(\frac{n-\ell_1}{r-\ell_1}\right)^s, \quad~\text{when $(\ell_1,s)\in \{(1,3),(2,2)\}$.}
\end{aligned}
\right.
$$

\subsection{The lower bound}\label{sec: LB 1}
We construct a graph $G_{\ell_1,s,1}$ in the following steps.
\begin{enumerate}
    \item Let $U_1, U_2,\ldots, U_r$ be the sets of independent vertices such that there are $\ell_1$ sets which only contain one vertex and each of the remaining sets contains $\lfloor \frac{n-\ell_1}{s} \rfloor$ vertices.
    \item Connect all edges between $U_i$ and $U_{i+1}, i=1,2,\ldots,r$, where we write $U_{r+1} = U_1$.
\end{enumerate}

If $r \ge 5$, then it is easy to verify all induced copies of $C_r$ in $G_{\ell_1,s,1}$ must have exactly one vertex from each $U_i, i=1,2,\ldots,r$. Therefore, $G_{\ell_1,s,1}$ has almost $n$ vertices and approximately $\left(\frac{n-\ell_1}{s} \right)^{s}$ induced copies of $C_r$.
Thus we have $\Psi_r(n, C_r, L) \ge (1+o(1))\left(\frac{n-\ell_1}{r-\ell_1}\right)^s$ when $r \ge 5$.

If $r=4$, then there must be $(\ell_1,s) \in \{(1,3),(2,2)\}$.
When $(\ell_1,s)=(2,2)$,  we let $|U_1|=|U_3|=1$. Then the graph $G_{\ell_1,s,1}$ becomes a complete bipartite graph $K_{2,n-2}$ which has $\binom{n-2}{2} = (2+o(1))\left(\frac{n-\ell_1}{r-\ell_1}\right)^s$ induced copies of $C_4$.

If $(\ell_1,s)=(1,3)$, assume $|U_1|=1$, we  add edges such that $G[U_3]$ forms a clique. Then all induced copies of $C_4$ must have  one of the following forms:
\begin{enumerate}
    \item $u_1u_2u_3u_4u_1, u_i\in U_i,i=1,2,3,4$, or
    \item $u_1u_2u_3u_2'u_1, u_1 \in U_1, u_3\in U_3, u_2,u_2'\in U_2$, or
    \item $u_1u_4u_3u_4'u_1, u_1 \in U_1, u_3\in U_3, u_4,u_4'\in U_4$.
\end{enumerate}
Then the graph  has $(2+o(1))\left(\frac{n-\ell_1}{r-\ell_1}\right)^s$ induced copies of $C_4$.

\subsection{The upper bound}\label{sec: UB 1}

Let $G$ be the extremal $(C_r, L)$-intersecting graph with $n$ vertices maximizing $N_{ind}(G,C_r)$.
That is, $\Psi_r(n,C_r,L) = N_{ind}(G,C_r)$.
By the low bound given in subsection \ref{sec: LB 1} and Theorem \ref{thm: deza-E-F}, we can assume there is $W\subseteq \bigcap_{A \in \mathscr{H}^r_{G,C_r}}A  $ such that $|W|=\ell_1$.
Let $\mathcal{S}, X_1$ be the corresponding sets defined in Section~\ref{sec: atom}.
By Lemma~\ref{lem: W not independent set psi},
 we may assume $W = \{w_1,w_2,\ldots,w_{\ell_1}\}$ is an independent set and $\ell_1 \le sd = s$.

\begin{lemma}\label{lem: entropy count cycle with fix w1}

If $s\geq 2$ and $Y$ is an independent set, then we have
\begin{equation*}
\begin{aligned}
    |\Upsilon(H,X,Y,\mathcal{S}',C_r)|\le \left\{ \begin{aligned}
        &\left( \frac{|\mathcal{S}'|}{s}\right)^{s}, \text{~if $s \ge 4$;}\\
        &2\left( \frac{|\mathcal{S}'|}{s}\right)^{s}, \text{~if $2 \le s \le 3$.}
    \end{aligned}  \right.
\end{aligned}
\end{equation*}
\end{lemma}

\noindent
\begin{proof}
Note that $\mathcal{S}'$ is the set of all singleton sets in $X$ by $d=1$.
Assume $Y = \{y_1,y_2,\ldots,y_{\ell_1}\}$.
We first define an operation called \textbf{vanishing}.
Let $Van(H)$ be the graph with vertex set $X\cup\{y_1\}$ and edge set
$$\begin{array}{rcl}
E(H[X\cup\{y_1\}]) &\cup &\{uv: \exists C \in \Upsilon(H, X, Y,  \mathcal{S}', C_r), \exists i \in \{2,\ldots,\ell_1\} \\
& &\text{~s.t.~}uy_i,vy_i \in E(H), u,v \in V(C)\cap X \}.
\end{array}$$
See Figure~\ref{fig: vanish} for an example when $r=12$ and $\ell_1=4$.
If $\ell_1=1$, then $Van(H) = H$.
Since $Y$ is an independent set, for each $C \in \Upsilon(H, X, Y, \mathcal{S}', C_r)$, it corresponds to a unique standard induced copy of $C_{1+sd}$ in $\Upsilon(Van(H), X, \{y_1\}, \mathcal{S}', C_{1+sd})$.
Thus, we have that
\begin{equation}\label{eq: vanish H}
\begin{aligned}
    |\Upsilon(H, X, Y,  \mathcal{S}', C_r)| \le & |\Upsilon(Van(H), X, \{y_1\}, \mathcal{S}', C_{1+sd})|.
\end{aligned}
\end{equation}

\begin{figure}[t]
	\centering
	\includegraphics[width=0.8\linewidth]{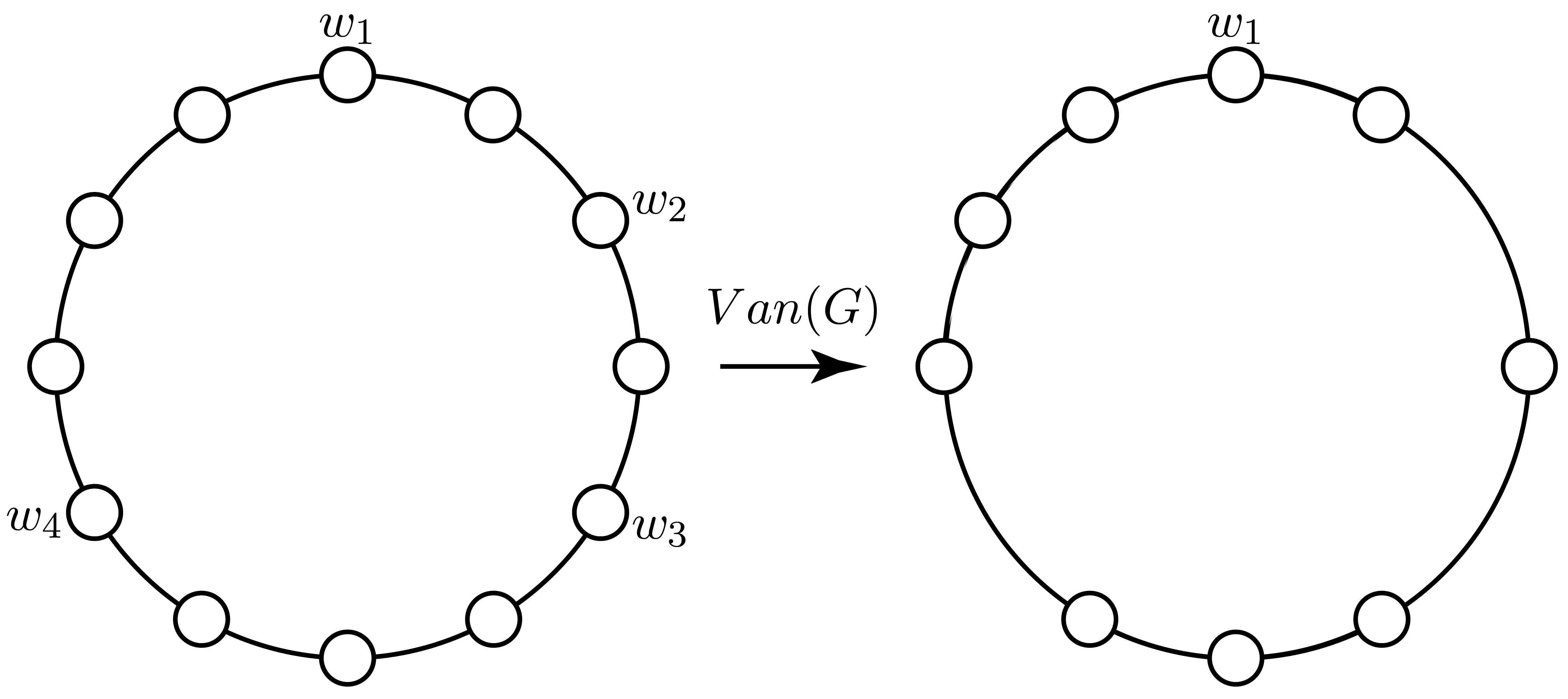}
	\caption{An example of an induced copy of $C_r$ after vanishing, when $r=12$ and $\ell_1=4$.}\label{fig: vanish}
\end{figure}

Hence, it suffices to prove the lemma when $\ell_1=1$.
In this case, $s=r-1$, $|\mathcal{S}'|=|V(H)|-1$ and
 $Y = \{y_1\}$.
    We use the entropy method to prove this.
    Note that every induced copies of $C_{r}$ in $\Upsilon(H,X,Y,\mathcal{S}',C_r)$ must contain $y_1$.
    We say an ordered vertex tuple $(u_1,u_2,u_3,\ldots,u_{r})$ is an \textbf{good tuple} if $u_1=y_1$ and $y_1u_2u_3\ldots u_{r}y_1$ is a standard induced copy of $C_r$ with $(X,Y,\mathcal{S}')$. Note that each $C \in \Upsilon(H,X,Y,\mathcal{S}',C_r)$ corresponds to exactly two good tuples.

    Let $Q = (Q_1, Q_2,\ldots, Q_{r})$ be a random variable chosen uniformly from all good tuples.
    By Propositions~\ref{prop: chain rule} and \ref{prop: entropy support bound}, the entropy of $Q$ is
    \begin{equation*}
        H(Q) = \log(2|\Upsilon(H,X,Y,\mathcal{S}',C_r)|) =H(Q_1) + H(Q_2\mid Q_1)+\ldots+H(Q_{r} \mid Q_{<r}),
    \end{equation*}
    where $Q_{<i}$ denotes $(Q_1,Q_2,\ldots,Q_{i-1})$ for $1 \le i \le r$.
    Note that $Q_1=y_1$ must hold. So by Proposition~\ref{prop: entropy support bound}, we have $H(Q_1) = 0$ and $H(Q_2\mid Q_1) \le \log |N_{H}(y_1)|$.

    For any $3 \le i \le r$, let $\beta(x_1,x_2,\ldots,x_{i-1})$ denote the number of good tuples $(u_1,u_2,\ldots,u_{r})$ such that $u_j=x_j$ for all $ 1 \le j \le i-1$.
    Let
    $$\begin{array}{rcl}
    \alpha(x_1,\ldots,x_{i-1})&=&\{y\in V(H) \mid \mbox{there exists a good tuple $(u_1,\ldots,u_{r})$ }\\
      & &\mbox{such that $u_j=x_j$ for all $ 1 \le j \le i-1$ and $u_i=y$ }\}.
    \end{array}$$
    Then we have
    \begin{equation*}
    \begin{aligned}
        H(Q_i\mid Q_{<i})& = \sum_{x_1,\ldots,x_{i-1} \in V(H)}\mathbb{P}(Q_j=x_j, j=1,\ldots,i-1) H(Q_i\mid Q_j=x_j,j=1,\ldots, i-1) \\
        & \le \sum_{x_1,\ldots,x_{i-1} \in V(H)} \frac{\beta(x_1,\ldots,x_{i-1})}{2|\Upsilon(H,X,Y,\mathcal{S}',C_r)|} \log |\alpha(x_1,\ldots,x_{i-1})|\\
        & = \sum_{x_1,\ldots,x_{i-1} \in V(H)} \sum_{(u_1,\ldots,u_{r})\text{is a good tuple}}\frac{\mathrm{1}_{u_j=x_j,j=1,2,\ldots,i-1}}{2|\Upsilon(H,X,Y,\mathcal{S}',C_r)|} \log |\alpha(x_1,\ldots,x_{i-1})| \\
        & = \sum_{(u_1,\ldots,u_{r})\text{ is a good tuple}} \sum_{x_1,\ldots,x_{i-1} \in V(H)} \frac{\mathrm{1}_{u_j=x_j,j=1,2,\ldots,i-1}}{2|\Upsilon(H,X,Y,\mathcal{S}',C_r)|} \log |\alpha(x_1,\ldots,x_{i-1})| \\
        & = \sum_{(u_1,\ldots,u_{r})\text{ is a good tuple}} \frac{1}{2|\Upsilon(H,X,Y,\mathcal{S}',C_r)|} \log |\alpha(u_1,\ldots,u_{i-1})|.\\
    \end{aligned}
    \end{equation*}
Thus,
\begin{equation}\label{eq: entropy 1}
\begin{aligned}
    H(Q) & = \log(2|\Upsilon(H,X,Y,\mathcal{S}',C_r)|) =H(Q_1) + H(Q_2\mid Q_1)+\ldots+H(Q_{r} \mid Q_{<r})\\
    & \le \log|N_{H}(y_1)|+\sum_{i=3}^{r}\sum_{(u_1,\ldots,u_{r})\text{ is a good tuple}} \frac{1}{2|\Upsilon(H,X,Y,\mathcal{S}',C_r)|} \log |\alpha(u_1,\ldots,u_{i-1})|\\
    & = \sum_{(u_1,\ldots,u_{r})\text{ is a good tuple}}\frac{1}{2|\Upsilon(H,X,Y,\mathcal{S}',C_r)|} \left(\log|N_{H}(y_1)|+  \sum_{i=3}^{r} \log |\alpha(u_1,\ldots,u_{i-1})| \right)\\
    & = \sum_{(u_1,\ldots,u_{r})\text{ is a good tuple}}\frac{1}{4|\Upsilon(H,X,Y,\mathcal{S}',C_r)|} \left(2\log|N_{H}(y_1)|+  \sum_{i=3}^{r} \log |\alpha_i^+| + \sum_{i=3}^{r} \log |\alpha_i^-| \right) \\
\end{aligned}
\end{equation}
where $\alpha_i^+ := \alpha(u_1,\ldots,u_{i-1})$ and $\alpha_i^- := \alpha(u_1, u_{r},\ldots,u_{r-i+3})$, $i=3,4,\ldots,r$.
Now we fix a good tuple $(u_1,u_2,\ldots,u_{r})$. First note that $\alpha_{r}^+, \alpha_{r}^- \subseteq N_{H}(y_1)$.
And $\{\alpha_i^+\}_{3 \le i \le r-1}$ are disjoint sets satisfying that $\alpha_i^+ \subseteq V(H) \setminus N_{H}[y_1]$, where $N_{H}[y_1]=N_{H}(y_1)\cup\{y_1\}$. Similarly, we have that $\{\alpha_i^-\}_{3 \le i \le r-1}$ are disjoint sets satisfying that $\alpha_i^- \subseteq V(H) \setminus N_{H}[y_1]$.

\textbf{Case 1: $s \ge 4$}.

In this case, we have that $\alpha_{r}^+ \cap \alpha_{r}^- = \emptyset$ by the definitions.
By AM-GM inequality, we have that
\begin{equation}\label{eq: entropy 2}
\begin{aligned}
    & 2\log\frac{|N_{H}(y_1)|}{2}+  \sum_{i=3}^{r} \log |\alpha_i^+| + \sum_{i=3}^{r} \log |\alpha_i^-|\\
    \le & \log \left( \frac{ |N_{H}(y_1)|+ \sum_{i=3}^{r} |\alpha_i^+| + \sum_{i=3}^{r} |\alpha_i^-|}{2(r-1)} \right)^{2(r-1)} \\
    \le & \log \left( \frac{ 2|N_{H}(y_1)|+ 2(|V(H)| - |N_{H}[y_1]|)}{2(r-1)} \right)^{2(r-1)}\\
    \le & \log \left( \frac{ |V(H)| - 1}{r-1} \right)^{2(r-1)}.
\end{aligned}
\end{equation}

Combining (\ref{eq: entropy 1}) and (\ref{eq: entropy 2}), we have that
\begin{equation*}\label{eq: entropy 3}
\begin{aligned}
    & H(Q) = \log(2|\Upsilon(H,X,Y,\mathcal{S}',C_r)|)\\
    \le & \sum_{(u_1,\ldots,u_{r})\text{ is a good tuple}} \frac{1}{4|\Upsilon(H,X,Y,\mathcal{S}',C_r)|}\left(2\log 2 + \log \left( \frac{ |V(H)| -1}{r-1} \right)^{2(r-1)} \right) \\
    = & \log 2 + \log \left( \frac{ |V(H)| - 1}{r-1} \right)^{(r-1)}\\
    = & \log 2 + \log \left( \frac{ |\mathcal{S}'|}{s} \right)^{s},
\end{aligned}
\end{equation*}
which proves the lemma in this case.

\textbf{Case 2: $s \in \{2,3\}$}.

In this case, we can not guarantee $\alpha_{r}^+ \cap \alpha_{r}^- = \emptyset$.
Therefore, we can only have the following estimation:
\begin{equation*}
\begin{aligned}
    & 2\log\frac{|N_{H}(y_1)|}{2}+  \sum_{i=3}^{r-1} \log |\alpha_i^+| + \sum_{i=3}^{r-1} \log |\alpha_i^-| + \log \frac{|\alpha_{r}^+|}{2} + \log \frac{|\alpha_{r}^-|}{2} \le \log \left( \frac{ |V(H)| - 1}{r-1} \right)^{2(r-1)},
\end{aligned}
\end{equation*}
by a similar argument with (\ref{eq: entropy 2}) and (\ref{eq: entropy 3}).
And it leads to the result that
\[
|\Upsilon(H,X,Y,\mathcal{S}',C_r)| \le 2\left( \frac{ |V(H)| - 1}{r-1} \right)^{(r-1)} \le 2\left( \frac{ |\mathcal{S}'|}{s} \right)^{s}.
\]
\end{proof}

By Lemma~\ref{lem: entropy count cycle with fix w1}, (\ref{eq: ub S}), (\ref{eq: Upsilon with scr H}) and (\ref{eq: only consider X1}), we can easily have that $\Psi_r(n,C_r,L) \le \left( \frac{n-\ell_1}{r-\ell_1} \right)^{s}$ when $s \ge 4$ and $\Psi_r(n,C_r,L) \le 2\left( \frac{n-\ell_1}{r-\ell_1} \right)^{s}$ when $(\ell_1,s) \in \{ (1,3), (2,2) \}$.
The unsolved situation is $(\ell_1,s)\in \{(2,3), (3,3)\}$, because we need to show that $|\Upsilon(H,X,Y,\mathcal{S}',C_r)|\le \left( \frac{n-\ell_1}{r-\ell_1} \right)^{s}$ in these two cases.


When $(\ell_1,s) = (3,3)$, let $W=\{w_1,w_2,w_3\}$, $V_1 = \{v \in V(G)\setminus W \mid vw_1,vw_2 \in E(G), vw_3 \notin E(G)\}$, $V_2 = \{v \in V(G)\setminus W \mid vw_2,vw_3 \in E(G), vw_1 \notin E(G)\}$ and $V_3 = \{v \in V(G)\setminus W \mid vw_3,vw_1 \in E(G), vw_2 \notin E(G)\}$.
Note that $V_1, V_2, V_3$ are pairwise disjoint sets. Since $W$ is an independent set, any induced copy of $C_6$ must be of the form $w_1u_1w_2u_2w_3u_3w_1$ where $u_i\in V_i, i=1,2,3$. Then the number of induced copies of $C_6$ is at most $|V_1||V_2||V_3| \le \left( \frac{n}{3} \right)^3 = (1+o(1))\left( \frac{n-\ell_1}{r-\ell_1} \right)^{s}$.

When $(\ell_1,s) = (2,3)$, let $W=\{w_1,w_2\}$, $V_1 = \{v \in V(G)\setminus W \mid vw_1,vw_2 \in E(G)\}$, $V_2 = \{v \in V(G)\setminus W \mid vw_2 \in E(G), vw_1 \notin E(G)\}$ and $V_3 = \{v \in V(G)\setminus W \mid vw_1\in E(G), vw_2 \notin E(G)\}$. Also, $V_1, V_2, V_3$ are pairwise disjoint sets, and any induced copy of $C_5$ must be of the form $w_1u_1w_2u_2u_3w_1$ where $u_i\in V_i, i=1,2,3$. Then the number of induced copies of $C_5$ is at most $|V_1||V_2||V_3| \le \left( \frac{n}{3} \right)^3 = (1+o(1))\left( \frac{n-\ell_1}{r-\ell_1} \right)^{s}$.

As a conclusion, now we entirely prove the upper bound when $d=1$.

\section{Proof of Theorem \ref{thm: the main theorem 2} (2)-(4)}\label{sec: proof AP 2-4}

In this section, $L=\{\ell_1,\ell_2,\dots,\ell_s\}\subseteq [1,r-1]$~($s > 1$) and $\ell_1,\ell_2,\dots,\ell_s,r$ form an arithmetic progression with common difference $d\ge 2$. That is, $\ell_{i+1}-\ell_i = d, i=1,2,\ldots,s$ and we make a deal that $\ell_{s+1}=r$. Then $r=\ell_1+ds$.

For convenience, we rewrite the result here.
\begin{enumerate}
    \item[(2)] If $d \ge 2$, then we first have a general upper and lower bound,
    $$ (1+o(1))\left(\frac{n-\ell_1}{r-\ell_1}\right)^s \le \Psi_r(n,C_r,L)\le (2+o(1))\left(\frac{n-\ell_1}{r-\ell_1}\right)^s.$$
    \item[(3)] If $d\geq 2$ and $\ell_1 > sd$, then
    $$ \Psi_r(n,C_r,L)= (1+o(1))\left(\frac{n-\ell_1}{r-\ell_1}\right)^s.$$
    \item[(4)] If $d\geq 2$, $s=2$, $\ell_1 < 2d$ and $ \ell_1 \in \{1,2\}$, then
    $$ \Psi_r(n,C_r,L)= (2+o(1))\left(\frac{n-\ell_1}{r-\ell_1}\right)^s.$$
\end{enumerate}

\subsection{ The lower bounds}\label{sec: LB 2}

When $d\geq 2$, the  graph $G_{\ell_1,s,d}$ is constructed as follows.

\begin{enumerate}
    \item Let $I \subseteq [\ell_1+s]$ be a set with size $|I| = \ell_1$.
    \item For $i \in I$, let $U_i=\{u_i\}$ and for $i \notin I$, let $U_i$ be the set of $k$ disjoint paths, say $\{\mathcal{P}_{i,1}, \mathcal{P}_{i,2}, \ldots, \mathcal{P}_{i,k}\}$ where $k = \lfloor \frac{n-\ell_1}{sd} \rfloor$ and $|V(\mathcal{P}_{i,j})|= d$ for $1\le j\le k$. Assume two end points of $\mathcal{P}_{i,j}$ is $u_{i,j}$ and $u'_{i,j}$.
    \item For each $i$ and $i+1$~(in modulo $\ell_1+s$),
    \begin{enumerate}
        \item if $i, i+1 \in I$, then connect $u_i$ and $u_{i+1}$;
        \item if $i\in I, i+1 \notin I$, then connect $u_i$ and $u_{i+1,j}$ for all $1 \le j \le k$;
        \item if $i\notin I, i+1 \in I$, then connect $u_{i+1}$ and $u'_{i,j}$ for all $1 \le j \le k$;
        \item if $i, i+1 \notin I$, then connect all $u'_{i,j_1}$ and $u_{i+1,j_2}$ for all $1 \le j_1,j_2 \le k$.
    \end{enumerate}
\end{enumerate}

An example when $\ell_1=2$, $s=3$ and $I = \{1,3\}$ is shown as Figure~\ref{fig: lb d ge 2 s ge 3}.
It is easy to verify that $G_{\ell_1,s,d}$ is $(C_r,L)$-intersecting with $(1+o(1))\left(\frac{n-\ell_1}{r-\ell_1}\right)^s$ induced copies of $C_r$. It proves that \[
\Psi_r(n,C_r,L) \ge  (1+o(1))\left(\frac{n-\ell_1}{r-\ell_1}\right)^s.
\]
\begin{figure}[t]
    \centering
    \includegraphics[width=0.5\linewidth]{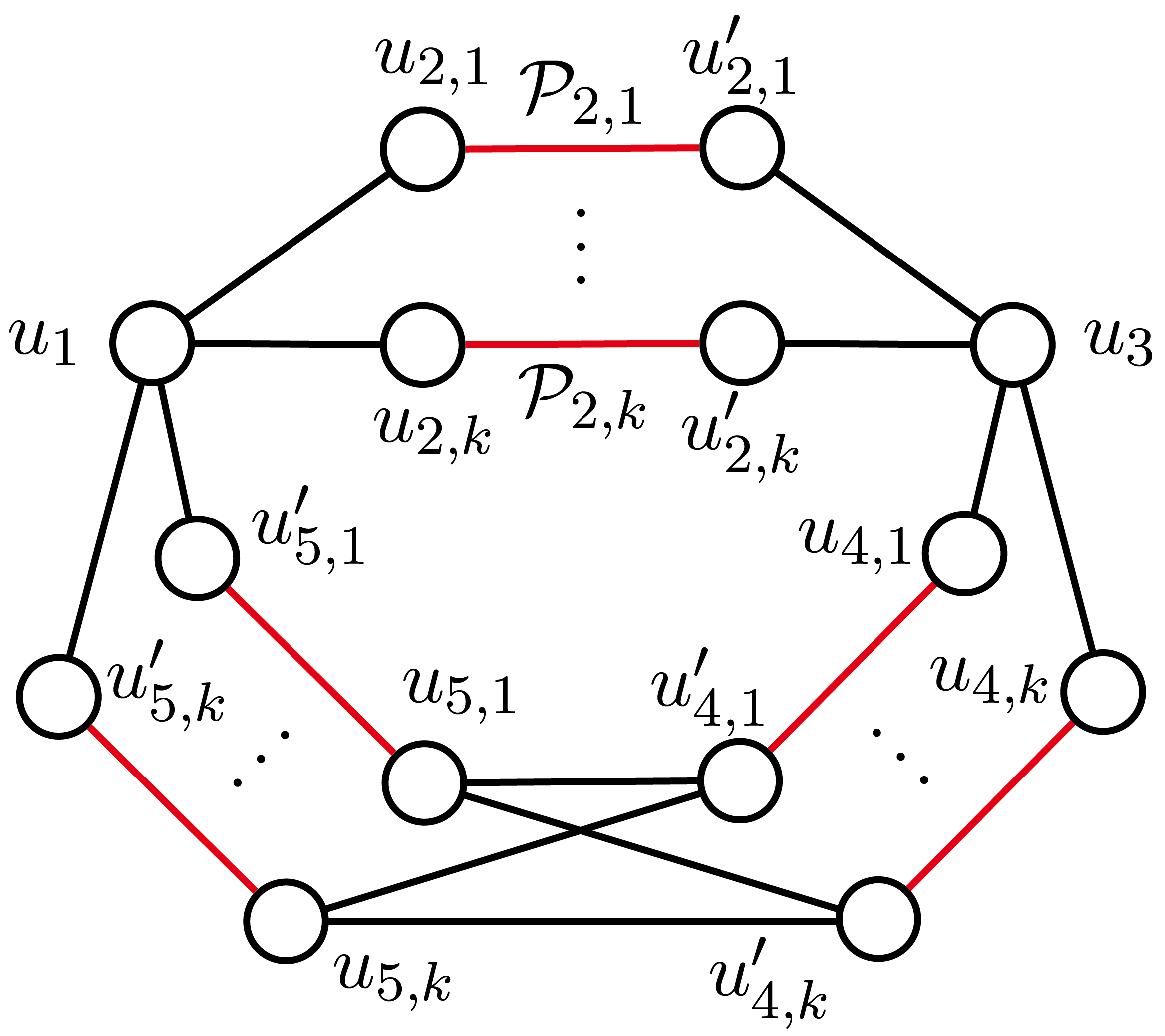}
    \caption{The sketch of $G_{\ell_1,s,d}$ when $\ell_1=2$, $d = 2$, $s=3$ and $I = \{1,3\}$. Each black line is an edge and each red line is a path with $d$ vertices. }
    \label{fig: lb d ge 2 s ge 3}
\end{figure}
So far, we have proved the lower bound of Theorem~\ref{thm: the main theorem 2}~(2) and (3).

For the lower bound of Theorem~\ref{thm: the main theorem 2}~(4) ($s=2$ and $\ell_1 =1$), let $I=\{1\}$ in the above construction and then let $\{u'_{2,j}\}_{1\le j \le k}$ and $\{u_{3,j}\}_{1\le j \le k}$ all form a clique. Then the result graph is $(C_r,L)$-intersecting with approximately $n$ vertices and $(2+o(1))\left(\frac{n-\ell_1}{r-\ell_1}\right)^s$ induced copies of $C_r$.
See Figure~\ref{fig: lb s=2 l=1} for an example.

For the lower bound of Theorem~\ref{thm: the main theorem 2}~(4) ($s=2$ and $\ell_1=2$), let $I = \{1,3\}$ in the above construction and then the result graph is $(C_r,L)$-intersecting with approximately $n$ vertices and $(2+o(1))\left(\frac{n-\ell_1}{r-\ell_1}\right)^s$ induced copies of $C_r$.
See Figure~\ref{fig: lb s=2 l=2} for an example.

\begin{figure}[t]
	\centering         
	\begin{minipage}{0.45\linewidth}
		\centering         
		\includegraphics[width=\linewidth]{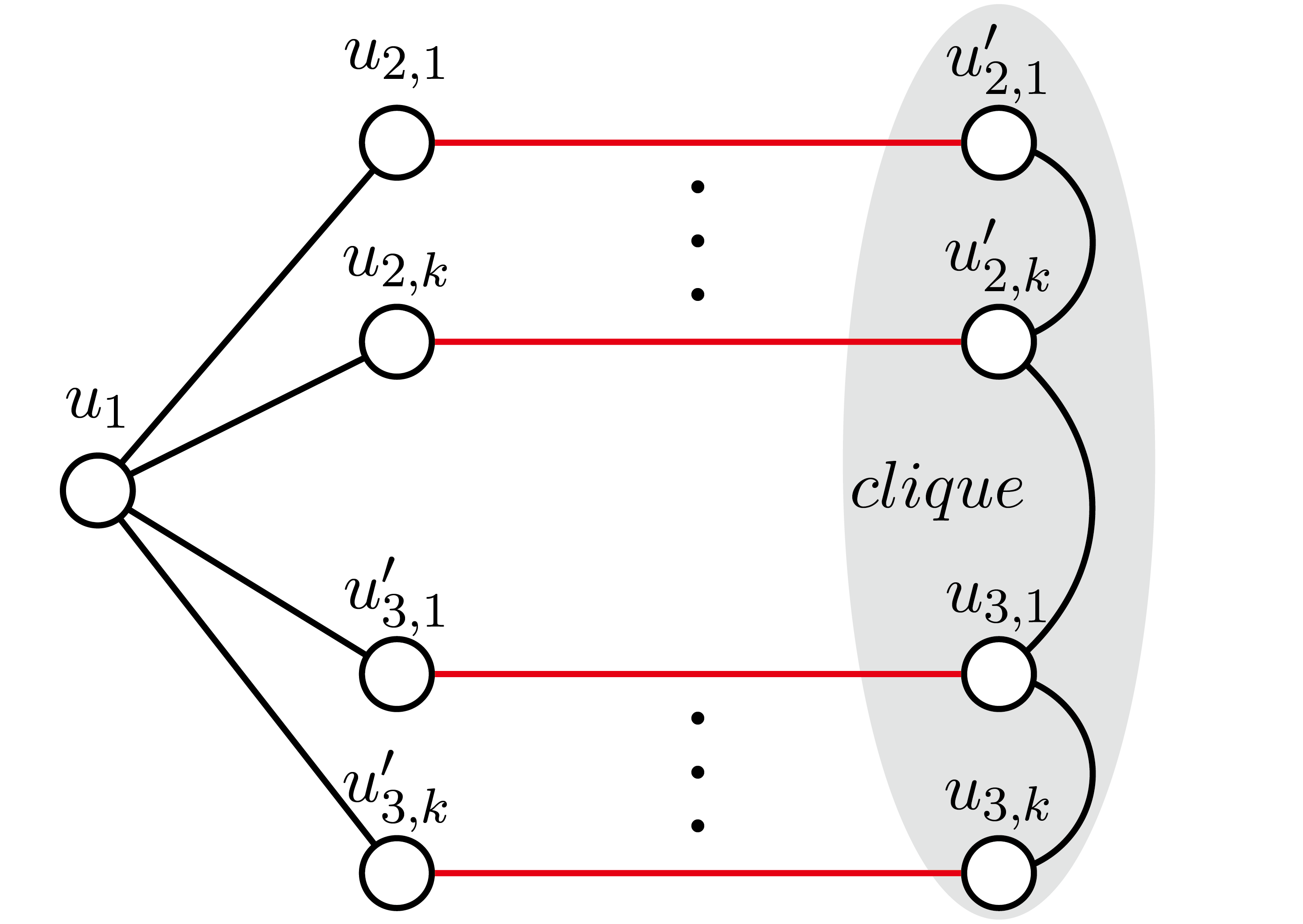}
		\caption{Extremal graph when $s=2$ and $\ell_1=1$. The grey part is a clique.}\label{fig: lb s=2 l=1}
	\end{minipage}
	\begin{minipage}{0.45\linewidth}
		\centering         
		\includegraphics[width=\linewidth]{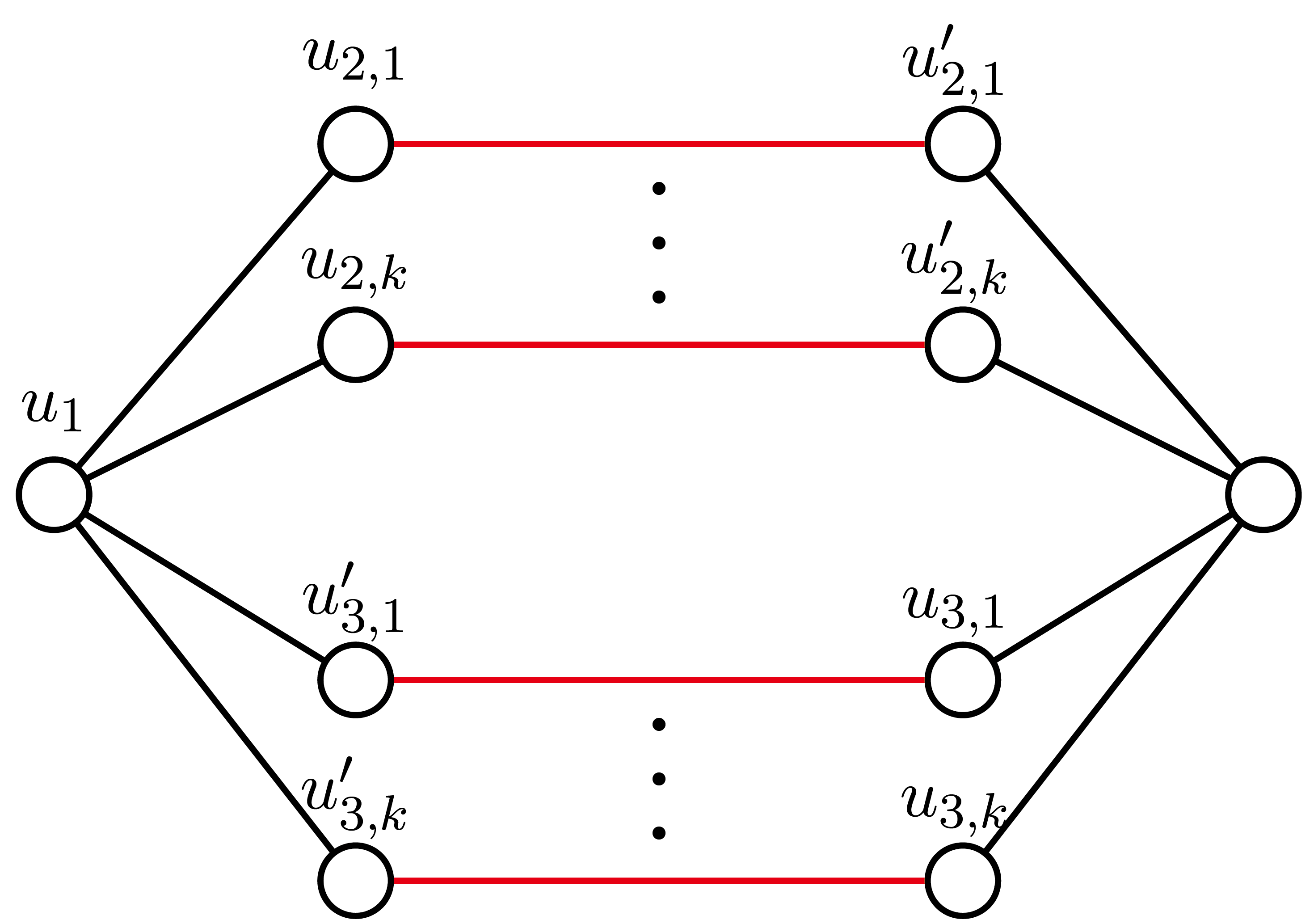}
		\caption{Extremal graph when $s=2$ and $\ell_1=2$.}\label{fig: lb s=2 l=2}
	\end{minipage}
\end{figure}

\subsection{The upper bounds}\label{sec: ub 2-4}

Let $G$ be the extremal $(C_r, L)$-intersecting graph with $n$ vertices maximizing $N_{ind}(G,C_r)$.
That is, $\Psi_r(n,C_r,L) = N_{ind}(G,C_r)$. By the low bound given in subsection \ref{sec: LB 2} and Theorem \ref{thm: deza-E-F}, we can assume there is $W\subseteq \bigcap_{A \in \mathscr{H}^r_{G,C_r}}A  $ such that $|W|=\ell_1$.
Let $\mathcal{S}, X_1$ be the corresponding sets defined in Section~\ref{sec: atom}.

The upper bound of Theorem~\ref{thm: the main theorem 2}~(3) comes from Lemma~\ref{lem: W not independent set psi}.
Hence, we only need to prove the upper bound of Theorem~\ref{thm: the main theorem 2}~(2).
By Lemma~\ref{lem: W not independent set psi}, it suffices to prove when $d\ge 2$, $W$ is an independent set and $\ell_1 \le sd$.
The following lemma is the key to prove the upper bound.
The idea is similar to Lemma~\ref{lem: entropy count cycle with fix w1}, but here we have an extra factor $2$ in the entropy estimation.

\begin{lemma}\label{lem: upper bound Upsilon}
    When $d\geq 2$, $s\geq 3$ and $Y$ is an independent set, $|\Upsilon(H,X,Y,\mathcal{S}',C_r)| \leq (2+o(1))\left(\frac{|\mathcal{S}'|}{s}\right)^s$.
\end{lemma}

\begin{proof}
We do the operation vanishing on $H$.
By (\ref{eq: vanish H}),  it suffices to prove the lemma when $\ell_1=1$.
Assume $Y = \{y_1\}$ and define

\begin{equation*}
    \begin{aligned}
        \mathcal{S}'_0(y_1):=&\{A\in \mathcal{S}':~ d_{G[Y\cup A]}(y_1)=0 \}, \\
        \mathcal{S}'_1(y_1):=&\{A\in \mathcal{S}':~ d_{G[Y\cup A]}(y_1)=1 \},\\
        \mathcal{S}'_2(y_1):=&\{A\in \mathcal{S}':~ d_{G[Y\cup A]}(y_1)=2 \}.
    \end{aligned}
\end{equation*}
Then $\mathcal{S}'_0(y_1)\cup \mathcal{S}'_1(y_1)\cup \mathcal{S}'_2(y_1) \subseteq \mathcal{S}'$.
For any $S \in \mathcal{S}' \setminus (\mathcal{S}'_0(y_1)\cup \mathcal{S}'_1(y_1)\cup \mathcal{S}'_2(y_1))$, no standard induced copy of $C_r$ in $\Upsilon(H,X,Y,\mathcal{S}',C_r)$ contains $S$.
Also, note that no induced cycle in $\Upsilon(H,X,Y,\mathcal{S}',C_r)$ contains two elements in $\mathcal{S}'_2(y_1)$; otherwise, $y_1$ has degree at least $4$ in an induced $C_r$, a contradiction. Similarly, there are no induced cycle in $\Upsilon(H,X,Y,\mathcal{S}',C_r)$ contains one element in $\mathcal{S}_1'(y_1)$ and one element in $\mathcal{S}_2'(y_1)$.

In the following discussion, taking $\ell_1, s,d, |\mathcal{S}'|$ as fixed constants, we may assume that $(H,X,Y,\mathcal{S}')$ maximizes $|\Upsilon(H,X,Y,\mathcal{S}',C_r)|$ and then minimizes $\min \{|\mathcal{S}_1'(y_1)|, |\mathcal{S}_2'(y_1)|\}$.
For a set $A\in \mathcal{S}'$, let $d_{\Upsilon(H,X,Y,\mathcal{S}',C_r)}(A)$ be the number of induced copies of $C_r$ in $\Upsilon(H,X,Y,\mathcal{S}',C_r)$ contained $A$.

\begin{claim}\label{claim: either S_1 or S_2}
    Either $\mathcal{S}'_1(y_1)=\emptyset$ or $\mathcal{S}'_2(y_1)=\emptyset$.
\end{claim}

\noindent
\textbf{Proof of Claim~\ref{claim: either S_1 or S_2}}
    Suppose $\mathcal{S}'_1(y_1)\neq \emptyset$ and $\mathcal{S}'_2(y_1)\neq \emptyset$.
    Let $A_0 \in \mathcal{S}'_1(y_1)\cup \mathcal{S}'_2(y_1)$ such that
    \[
    d_{\Upsilon(H,X,Y,\mathcal{S}',C_r)}(A_0)=\max\{d_{\Upsilon(H,X,Y,\mathcal{S}',C_r)}(A):~A\in \mathcal{S}'_1(y_1)\cup \mathcal{S}'_2(y_1)\}.
    \]

    \textbf{Case 1: $A_0\in \mathcal{S}'_2(y_1)$}.
    Since there is no  $C\in \Upsilon(H,X,Y,\mathcal{S}',C_r)$ contained two atoms in $\mathcal{S}'_1(y_1)$ and $\mathcal{S}'_2(y_1)$ respectively, we can construct a new graph $H'$ from $H$ by repeating $A_0$ $|\mathcal{S}'_1(y_1)|$ times and then deleting $\bigcup_{S \in \mathcal{S}_1'(y_1)}S$. Assume the repeated copies of $A_0$ in $H'$ are denoted by $A_0^{(1)}, A_0^{(2)}, \ldots, A_0^{(|\mathcal{S}'_1(y_1)|)}$.
    Here repeating $A_0$ by $A_0^{(i)}$~($1\le i\le  |\mathcal{S}'_1(y_1)|$) means there exists a bijection $\phi: A_0 \to A_0^{(i)}$ such that $\phi(x)\phi(y) \in E(H')$ if and only if $xy\in E(H)$ for any $x,y \in A_0$, and $\phi(x)u \in E(H')$ if and only if $xu \in E(H)$ for any $x \in A_0, u \in (X\cup U) \setminus A_0$. Moreover, there is no edges between $A_0^{(i)}$ and $A_0^{(j)}$ for any $0 \le i < j \le |\mathcal{S}'_1(y_1)|$ where $A_0^{(0)} := A_0$.
    Write $Z_1 = \bigcup_{S \in \mathcal{S}_1'(y_1)}S$ and $Z_2 = \bigcup_{i=1}^{|\mathcal{S}'_1(y_1)|} A_0^{(i)}$.
    Then $H'$ has vertex set $ (X \setminus Z_1) \cup Z_2 \cup Y$ and \[
    \mathcal{S}'' := \mathcal{S}'_0(y_1) \cup \mathcal{S}'_2(y_1) \cup \{A_0^{(i)}\}_{1 \le i \le |\mathcal{S}'_1(y_1)|}
    \]
    is a partition of $(X \setminus Z_1) \cup Z_2$. Note that $|\mathcal{S}''| = |\mathcal{S}'|$.
    Now we consider the set $\Upsilon(H', (X \setminus Z_1) \cup Z_2, Y, \mathcal{S}'', C_r)$.
    For each $C \in \Upsilon(H,X,Y,  \mathcal{S}', C_r)$, if $C$ contains no element from $\mathcal{S}_1'(y_1)$, then $C \in \Upsilon(H', (X \setminus Z_1) \cup Z_2, Y, \mathcal{S}'', C_r)$.
    If $C$ contains some elements from $\mathcal{S}_1'(y_1)$, then $C$ is not in $\Upsilon(H', (X \setminus Z_1) \cup Z_2, Y, \mathcal{S}'', C_r)$. There are at most $d_{\Upsilon(H,X,Y,\mathcal{S}',C_r)}(A_0)\cdot |\mathcal{S}_1'(y_1)|$ such $C$ contained elements from $\mathcal{S}_1'(y_1)$.
    For each $A_0^{(i)}$, there will arise at least $d_{\Upsilon(H,X,Y,\mathcal{S}',C_r)}(A_0)$ new elements in $\Upsilon(H', (X \setminus Z_1) \cup Z_2, Y, \mathcal{S}'', C_r)$ by replacing $A_0$ with $A_0^{(i)}$ in every $C \in \Upsilon(H,X,Y,  \mathcal{S}', C_r)$ contained $A_0$.
    As a result, we have that $|\Upsilon(H', (X \setminus Z_1) \cup Z_2, Y, \mathcal{S}'', C_r)| \ge |\Upsilon(H,X,Y,  \mathcal{S}', C_r)|$. By the maximality when we choose $(H,X,Y,\mathcal{S}')$, the equality holds. But $H'$ admits a lower value of
    $
    \min \{|\mathcal{S}_1''(y_1)|, |\mathcal{S}_2''(y_1)|\} = 0,
    $ a contradiction.

    \textbf{Case 2: $A_0\in \mathcal{S}'_1(y_1)$}. We can do a similar repeating of $A_0$ but removing $\bigcup_{S \in \mathcal{S}_2'(y_1)}S$. Then we can obtain the result. The details are omitted here.
\hfill $\blacksquare$ \par

If $|\mathcal{S}_1'(y_1)|=0$, then each $C \in \Upsilon(H,X,Y,  \mathcal{S}', C_r)$ contains exactly one element from $\mathcal{S}_2'(y_1)$ and $(s-1)$ elements from $\mathcal{S}_0'(y_1)$.
Now fix an element $S' \in \mathcal{S}_2'(y_1)$.
Write $Z = \bigcup_{S \in \mathcal{S}_2'(y_1)}S$.

Then each $C \in \Upsilon(H,X,Y,  \mathcal{S}', C_r)$ contained $S'$ is also in $\Upsilon(H[X\setminus  Z \cup Y \cup S'], X\setminus Z, Y\cup S',\mathcal{S}_0'(y_1),C_r)$.
Now $E(H[Y \cup S']) \neq \emptyset$.
By Lemma~\ref{lem: e HW has edge, Upsilon}, we have that $|\Upsilon(H[X\setminus  Z \cup Y \cup S'], X\setminus Z, Y\cup S',\mathcal{S}_0'(y_1),C_r)| \le (1+o(1))\left( \frac{|\mathcal{S}_0'(y_1)|}{s-1} \right)^{s-1}$.
Then we have that
\begin{equation*}
\begin{aligned}
    |\Upsilon(H,X,Y,  \mathcal{S}', C_r)| & \le |\mathcal{S}_2'(y_1)| \cdot (1+o(1))\left( \frac{|\mathcal{S}_0'(y_1)|}{s-1} \right)^{s-1} \\
    & \le (1+o(1))|\mathcal{S}_2'(y_1)| \left( \frac{|\mathcal{S}'| - |\mathcal{S}_2'(y_1)|}{s-1} \right)^{s-1} \\
    & \le (1+o(1))\left( \frac{|\mathcal{S}'|}{s} \right)^{s}\\
    & = (1+o(1))\left( \frac{|\mathcal{S}|}{s} \right)^{s},
\end{aligned}
\end{equation*}
which proves the upper bound. So by Claim~\ref{claim: either S_1 or S_2}, we only need to consider the case when $|\mathcal{S}_2'(y_1)|=0$.

Note that each $C \in \Upsilon(H,X,Y,  \mathcal{S}', C_r)$ can be decomposed into $V(C) = Y \cup \left( \bigcup_{i=1}^{s}S_i \right)$, where $S_i \in \mathcal{S}', i=1,2,\ldots,s$ and there exactly two $S_{j_1},S_{j_2} \in \mathcal{S}_1'(y_1)$ while the rest $S_j$ is in $\mathcal{S}_0'(y_1)$.

Now we define a total order $\sigma$ on $X$,
that is, a bijection $\sigma$ from $X$ to $\{1,2,\ldots,|X|\}$.
We abbreviate $\sigma(u) < \sigma(v)$ to $u <_{\sigma} v$.
If $(S_1,S_2,S_3,\ldots,S_s)$ satisfies the following properties, we say $(S_1,S_2,S_3,\ldots,S_s)$ is a \textbf{good decomposition} of $C$.
\begin{enumerate}
    \item $S_1,S_2 \in \mathcal{S}_1'(y_1)$ and $z_1<_{\sigma} z_2$, where $\{z_i\} = S_i\cap N_H(y_1)$, $i=1,2$.
    \item $S_3, \ldots, S_s \in \mathcal{S}_0'(y_1)$. For $i\in\{2,\ldots, s-1\}$, in $H[Y \cup S_1 \cup \ldots\cup S_i]$, there exists a vertex $x_i$ of degree one and a path from $y_1$ to $x_i$ via $z_1$. Such a vertex always exists since $C$ is a cycle. Then $x_i$ must have a neighbor in $S_{i+1}$. Since $C$ is an induced copy of $C_r$, it follows that $x_i$ has no neighbors in $S_{j}$ for any $j>i+1$.
\end{enumerate}

Note that each $C \in \Upsilon(H, X, Y,  \mathcal{S}', C_r)$ corresponds to exactly one good decomposition.
Let $\mathcal{D}$ be the collection of good decompositions of  $\Upsilon(H, X, Y,  \mathcal{S}', C_r)$ and let $Q = (Q_1, Q_2,\ldots, Q_s)$ be a random variable chosen uniformly from $\mathcal{D}$. Then the entropy of $Q$ is $H(Q) = \log |\mathcal{D}| = \log (|\Upsilon(H, X, Y,  \mathcal{S}', C_r)|)$ by Proposition~\ref{prop: entropy support bound}.
By Proposition~\ref{prop: chain rule}, we have that
\begin{equation*}
\begin{aligned}
    H(Q) & = H(Q_1, Q_2) + \sum_{i=3}^{s}H(Q_i \mid Q_{<i}),
\end{aligned}
\end{equation*}
where $Q_{<i} := (Q_1,Q_2,\ldots,Q_{i-1})$ for $i\in\{2,\ldots,s\}$. Then we will do similar work as in Section~\ref{sec: UB 1}.
Note that $Q_1,Q_2 \in \mathcal{S}'_1(y_1)$ and $z_1 <_{\sigma} z_2$ must be fulfilled, where $\{z_i\}= Q_i\cap N_H(y_1)$ for $i=1,2$.
Then the support of $(Q_1,Q_2)$ has size at most $\binom{|\mathcal{S}'_1(y_1)|}{2} \le |\mathcal{S}'_1(y_1)|^2/2$ since after choosing a pair from $\mathcal{S}'_1(y_1)$, the order is determined by checking whether $z_1 <_{\sigma} z_2$ holds.
Therefore, $H(Q_1) \le \log (|\mathcal{S}'_1(y_1)|^2/2)$.
For a sequence $(P_1,P_2,\ldots,P_{i-1})$ where $P_j \in \mathcal{S}', 1 \le j \le i-1$, let $\beta(P_1,P_2,\ldots,P_{i-1})$ denote the number of $(S_1,\ldots,S_s) \in \mathcal{D}$ such that $S_j=P_j,\forall j=1,2,\ldots,i-1$. Let $\alpha(P_1,P_2,\ldots,P_{i-1})$ denote the set of $P_i \in \mathcal{S}'$ such that there exists $(S_1,\ldots,S_s) \in \mathcal{D}$ such that $S_j=P_j,\forall j=1,2,\ldots,i$.
For $i=3,\ldots,s$, we have that
\begin{equation*}
\begin{aligned}
    H(Q_i \mid Q_{<i}) &= \sum_{\{P_j\}_{j=2,\ldots,i-1} \subseteq \mathcal{S}'} \mathbb{P}(Q_j=S_j,j=1,\ldots,i-1) H(Q_i \mid Q_j=S_j,j=1,\ldots,i-1) \\
    & \le \sum_{\{P_j\}_{j=2,\ldots,i-1} \subseteq \mathcal{S}'} \frac{\beta(P_1,\ldots,P_{i-1})}{|\mathcal{D}|} \log |\alpha(P_1,\ldots,P_{i-1})| \\
    & \le \sum_{\{P_j\}_{j=2,\ldots,i-1} \subseteq \mathcal{S}'} \sum_{(S_1,S_2,\ldots,S_s) \in \mathcal{D}} \frac{\mathrm{1}_{S_j=P_j,j=1,2,\ldots,i-1}}{|\mathcal{D}|} \log |\alpha(P_1,\ldots,P_{i-1})| \\
    & \le \sum_{(S_1,S_2,\ldots,S_s) \in \mathcal{D}} \sum_{\{P_j\}_{j=2,\ldots,i-1} \subseteq \mathcal{S}'}  \frac{\mathrm{1}_{S_j=P_j,j=1,2,\ldots,i-1}}{|\mathcal{D}|} \log |\alpha(P_1,\ldots,P_{i-1})| \\
    & \le \sum_{(S_1,S_2,\ldots,S_s) \in \mathcal{D}} \frac{1}{|\mathcal{D}|} \log |\alpha(S_1,\ldots,S_{i-1})|,
\end{aligned}
\end{equation*}
and

\begin{equation}\label{eq: entropy d>=2 1}
\begin{aligned}
    H(Q) & = \log|\mathcal{D}| = H(Q_1,Q_2) + \sum_{i=3}^{s}H(Q_i \mid Q_{<i})\\
    & \le \sum_{(S_1,S_2,\ldots,S_s) \in \mathcal{D}} \frac{1}{|\mathcal{D}|} \left(\log (|\mathcal{S}'_1(y_1)|^2/2) + \sum_{i=3}^{s}  \log |\alpha(S_1,\ldots,S_{i-1})| \right). \\
\end{aligned}
\end{equation}

Now we fix $(S_1,S_2,\ldots,S_s) \in \mathcal{D}$ and write $\alpha_i = \alpha(S_1,\ldots,S_{i-1})$ for $i\in\{3,\ldots,s\}$.
Note that, for $i\in \{3,\ldots,s\}$, $\alpha_i \subseteq \mathcal{S}_0'(y_1)$ and $\{\alpha_i\}_{i=3,\ldots,s}$ are disjoint sets by the second rule of good decomposition.
By AM-GM inequality, we have that

\begin{equation}\label{eq: entropy d>=2 2}
\begin{aligned}
    &2\log \frac{|\mathcal{S}_1'(y_1)|}{2} + \sum_{i=3}^{s}\log |\alpha_i| \\
    \le & \log \left(\frac{|\mathcal{S}_1'(y_1)| + \sum_{i=2}^{s}|\alpha_i| }{s} \right)^{s} \\
    \le & \log \left(\frac{|\mathcal{S}_1'(y_1)| + |\mathcal{S}_0'(y_1)| }{s} \right)^{s} \\
    = & \log \left(\frac{|\mathcal{S}'| }{s} \right)^{s}. \\
\end{aligned}
\end{equation}

Combining (\ref{eq: entropy d>=2 1}) and (\ref{eq: entropy d>=2 2}), we have that
\begin{equation*}
\begin{aligned}
    & H(Q) = \log|\mathcal{D}| = \log (|\Upsilon(H, X, Y,  \mathcal{S}', C_r)|) \\
    \le & \sum_{(S_1,S_2,\ldots,S_s) \in \mathcal{D}} \frac{1}{|\mathcal{D}|} \left(\log (|\mathcal{S}'_1(y_1)|^2/2) + \sum_{i=3}^{s}  \log |\alpha(S_1,\ldots,S_{i-1})| \right) \\
    =& \log2 + \log \left(\frac{|\mathcal{S}'| }{s} \right)^{s},
\end{aligned}
\end{equation*}
which derives that $|\Upsilon(H, X, Y,  \mathcal{S}', C_r)| \le 2\left(\frac{|\mathcal{S}'| }{s} \right)^{s}$.
Recalling that $|\mathcal{S}'| = |\mathcal{S}|$, then the lemma is proved when $\ell_1=1$.

\end{proof}

Combining (\ref{eq: Upsilon with scr H}) and (\ref{eq: only consider X1}), Lemma~\ref{lem: upper bound Upsilon}, we can prove the upper bound easily.

\section{Proof of Theorem~\ref{thm: the main theorem 2}~(5)}\label{sec: 5}

The upper bound of Theorem~\ref{thm: the main theorem 2}~(5) comes from Theorem~\ref{thm: the main theorem 2}~(2). We only need to construct a graph that reaches the upper bound.

In this section, $L=\{\ell_1,\ell_2\}\subseteq [1,r-1]$ and $\ell_1,\ell_2,r$ form an arithmetic progression with common difference $d\ge 2$. That is, $\ell_{2}-\ell_1 = r-\ell_2 = d$.
For convenience, we write $\ell := \ell_1$ and then $\ell_2 = \ell +d, r=\ell+2d$. Moreover, we require $\ell < d/2$ and $\ell \in \{3,4\}$.
We then construct the  graph $G_{\ell,d}$ in the following steps.

\begin{enumerate}
    \item Write $W = \{w_1,w_2,\ldots,w_\ell\}$.
    \item Let $k = \lfloor (n-\ell)/d \rfloor$. For each $w_i$ ($1\le i\le \ell$), add $k$ new vertices $v_1^{(i)}, v_2^{(i)}, \ldots, v_k^{(i)}$ and for each $j$ ($1\le j\le k$) and $i$ ($1\le i\le \ell$), connect $w_i$ and $v_j^{(i)}$ by a path $\mathcal{P}_{ij}$. Denote its length by $|\mathcal{P}_{ij}|$,  the number of edges of $\mathcal{P}_{ij}$. Let $|\mathcal{P}_{ij}| = p_i$ for each $j$. We require that $\sum_{i=1}^{l}p_i = d$ and $p_i \ge 2$ for each $i$.
    \item For $t$ and $t+1$~(in modulo $\ell$) ($1\le t\le \ell$), if $j > i$, then connect $v_{j}^{(t)}$ and $v_{i}^{(t+1)}$.
\end{enumerate}

\begin{figure}[t]
	\centering
	\includegraphics[width=0.7\linewidth]{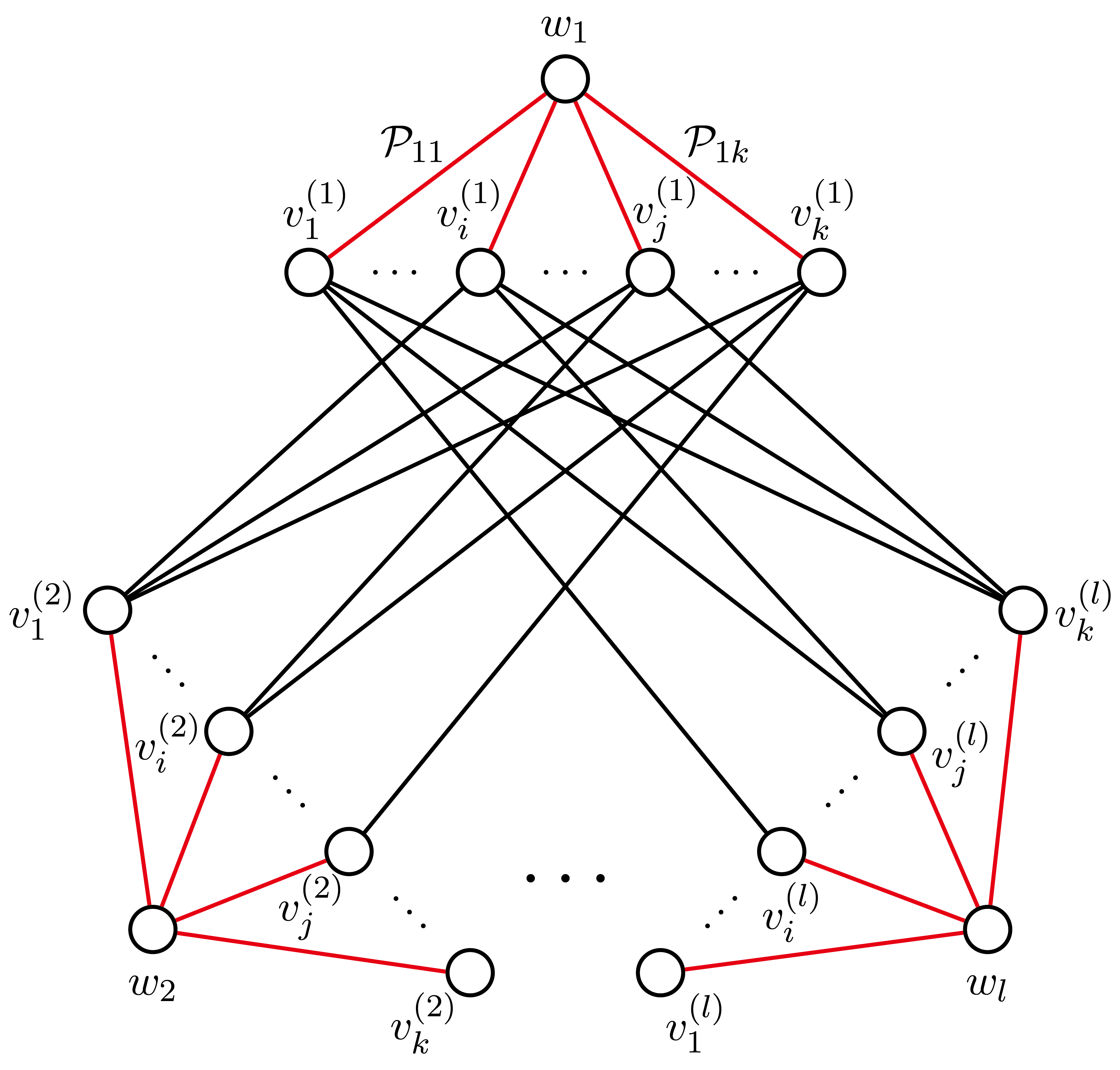}
	\caption{The red lines represent paths.  }\label{fig: lb, d ge 2l, l ge 3}
\end{figure}

Figure~\ref{fig: lb, d ge 2l, l ge 3} shows an example.
Note that $w_1 \mathcal{P}_{1i} \mathcal{P}_{2j}w_2 \mathcal{P}_{2i} \mathcal{P}_{3j} w_3 \ldots \mathcal{P}_{\ell i}w_\ell  \mathcal{P}_{\ell j} \mathcal{P}_{1i}w_1$ is an induced $C_{\ell+2d}$ for each $1 \le i < j \le k$ and we denote it by $C_r^{(ij)}$.
It is easy to see that $C_r^{(ij)}$ and $C_r^{(i'j')}$ have either $\ell$ or $\ell+d$ vertices in common for each $(i,j)\neq (i',j')$.
Then $G_{\ell,d}$ has $\ell + kd \le n$ vertices and at least $\binom{k}{2} = (1+o(1))\frac{n^2}{2d^2}$ induced copies of $C_{r}$.
If we can prove that there are no other induced copies of $C_r$ in $G_{\ell,d}$, then $\mathscr{H}_{G, C_r}^{r}$ is $\{\ell, \ell+d\}$-intersecting and thus, $\Psi_r(n, C_r, \{\ell,\ell+d\}) \ge (1+o(1))\frac{n^2}{2d^2}$.

When $\ell \in \{3,4\}$, it is easy to verify that there are no other induced copies of $C_r$, and we omit the details here.
But when $\ell$ is large enough, it would certainly fail since there would arise other induced copies of $C_r$.
We believe when $\ell$ is large, the factor before $\left( \frac{n-\ell_1}{r-\ell_1} \right)^s$ in $\Psi_r(n, C_r, L)$ will be strictly less than two.
In previous cases, the behaviour of $\Psi_r(n,C_r,L)$ only depends on $s$ and $d$.
In this case, $\ell_1$ plays an important role in deciding the asymptotic behavior of $\Psi_r(n, C_r, L)$, and it may reveal some inherent difficulties in the case when $s=2$.

\section{Proof of Theorem~\ref{thm: the main theorem 2} (6)} \label{sec: 6}

In this section, $L=\{\ell_1,\ell_2\}\subseteq [1,r-1]$ and $\ell_1,\ell_2,r$ form an arithmetic progression with common difference $d\ge 2$. That is, $\ell_{2}-\ell_1 = r-\ell_2 = d$. Moreover, we require $\ell_1 = 2d$.
For convenience, we write $\ell := \ell_1$ and then $\ell_2 = \ell +d, r=\ell+2d=4d$ and rewrite the result here:

$$(6)~\left(\frac{4}{3} +o(1)\right) \left(\frac{n-\ell_1}{r-\ell_1}\right)^s \le \Psi_r(n, C_r, L) \le \left(2 - \frac{2}{f(\ell_1)} +o(1)\right) \left(\frac{n-\ell_1}{r-\ell_1}\right)^s,$$

or equivalently,
$$(6)~\left(\frac{2}{3} +o(1)\right) \frac{n^2}{2d^2} \le \Psi_r(n, C_r, L) \le \left(1 - \frac{1}{f(\ell)} +o(1)\right) \frac{n^2}{2d^2}.$$

Before proving the theorem, we first prove  a lower bound of $f(\ell)$.
\begin{lemma}\label{lem: lb f(n)}
    $f(\ell) \ge 3$ for even $\ell \ge 4$.
\end{lemma}

\begin{proof}
    We prove the lower bound by constructing a graph with $\ell$ vertices that admits a flawless $1$-factorization for each even $\ell \ge 4$.
    Let $U = \{u_1,u_2,\ldots,u_\ell\}$, $M_1 = \{u_1u_2,u_3u_4,\ldots,u_{\ell-1}u_{\ell}\}$ and $M_2 = \{u_2u_3,u_4u_5,\ldots,u_{\ell-2}u_{\ell-1},u_{\ell}u_1\}$.

    If $\ell=4$, let $M_3 = \{u_1u_3,u_2u_4\}$; if $\ell \ge 6$, then let
    \[M_3 = \{u_1u_3, u_4u_6\} \cup \{u_iu_{i+3}\}_{i=5,7,9,\ldots,\ell-3} \cup \{u_{\ell-1}u_2\}.\]
    Let $F$ be the graph with vertex set $U$ and edge set $M_1\cup M_2\cup M_3$.
    Figure~\ref{fig: f(n) ge 3} shows the example when $\ell=4$ and $\ell=12$.
    It is easy to verify $\{M_1,M_2,M_3\}$ is a flawless $1$-factorization of $F$ by its definition.
    We omit the details here.
\end{proof}

\begin{figure}[t]
	\centering
	\includegraphics[width=0.7\linewidth]{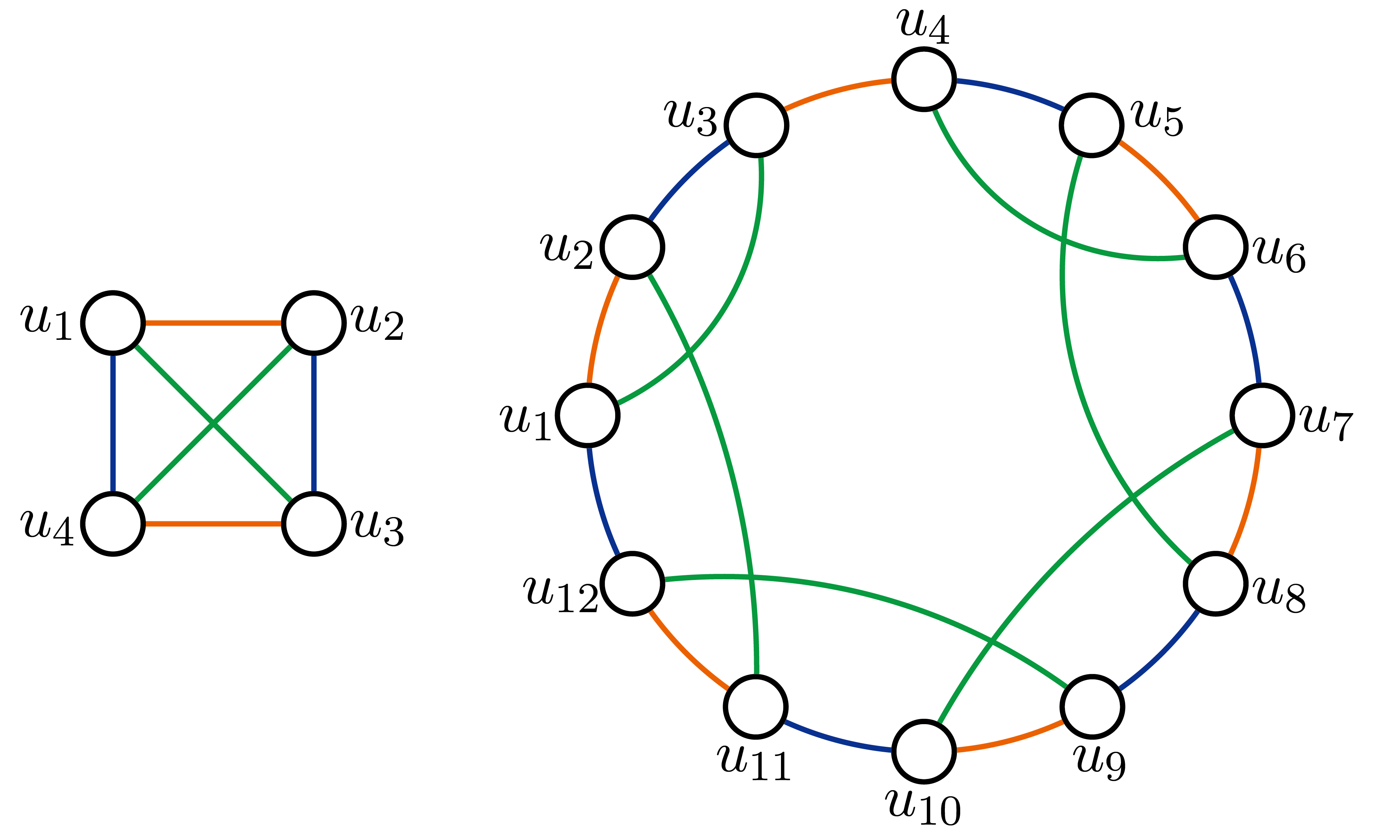}
	\caption{$M_1$ is the set of orange edges, $M_2$ is the set of blue edges and $M_3$ is the set of green edges. $F$ is the graph with edge set $M_1\cup M_2 \cup M_3$ and $\{M_1,M_2,M_3\}$ is a flawless $1$-factorization of $F$.}\label{fig: f(n) ge 3}
\end{figure}

\subsection{The lower bound}\label{sec: lb 6}

Note that $r=4d \ge 8$.
By Lemma~\ref{lem: lb f(n)}, let $\{M_1,M_2,M_3\}$ be a flawless $1$-factorization of a graph $F$ with vertices $\{u_1,\ldots,u_{\ell}\}$.
We construct a graph as follows.
\begin{enumerate}
    \item Let $F'$ be the sub-division of $F$, that is, for each edge $u_iu_j \in E(F)$, add a new vertex $v_{u_iu_j}$, remove the edge $u_iu_j$ and then add two edges $u_iv_{u_iu_j}, u_jv_{u_iu_j}$.
    \item Replace each $v_{u_iu_j}$ by a $k$-clique with vertices $\{v_{u_iu_j}^{t}\}_{1 \le t \le k}$ where $k = \lfloor \frac{n-\ell}{3d} \rfloor$, and if $v' \in N(v_{u_iu_j})$, then connect $v'v_{u_iu_j}^{t}$ for all $t=1,2,\ldots,k$.
    \item Note that each $\{v_{u_iu_j}^{t}\}_{1 \le t \le k}$ corresponds to an edge $u_iu_j$ in $F$, and each edge in $F$ corresponds to $k$ new vertices. If  the edges $u_{i_1}u_{j_1}$ and $u_{i_2}u_{j_2}$  belong to the same matching $M' \in \{M_1,M_2,M_3\}$ in $F$, then we connect $u_{i_1j_1}^{t_1}$ and $u_{i_2j_2}^{t_2}$ for all $1\le t_1,t_2\le k$ but $t_1 \neq t_2$.
\end{enumerate}

Denote the result graph by $F''$.
Define
\begin{equation*}
\begin{aligned}
    &E_1 = \{v_e^{t_1}v_e^{t_2} \mid e\in E(F), t_1\neq t_2\}, \\
    &E_2 = \{v_{e_1}^{t}v_{e_2}^{t} \mid \exists i, \text{s.t.}~e_1,e_2 \in M_i \}. \\
\end{aligned}
\end{equation*}
Then
$E_1$ denotes the edges in the $k$-clique and $E_2$ denotes the edges added in step 3.
See Figure~\ref{fig: lb flawless} for an example when $\ell=4$.

\begin{figure}[t]
	\centering
	\includegraphics[width=\linewidth]{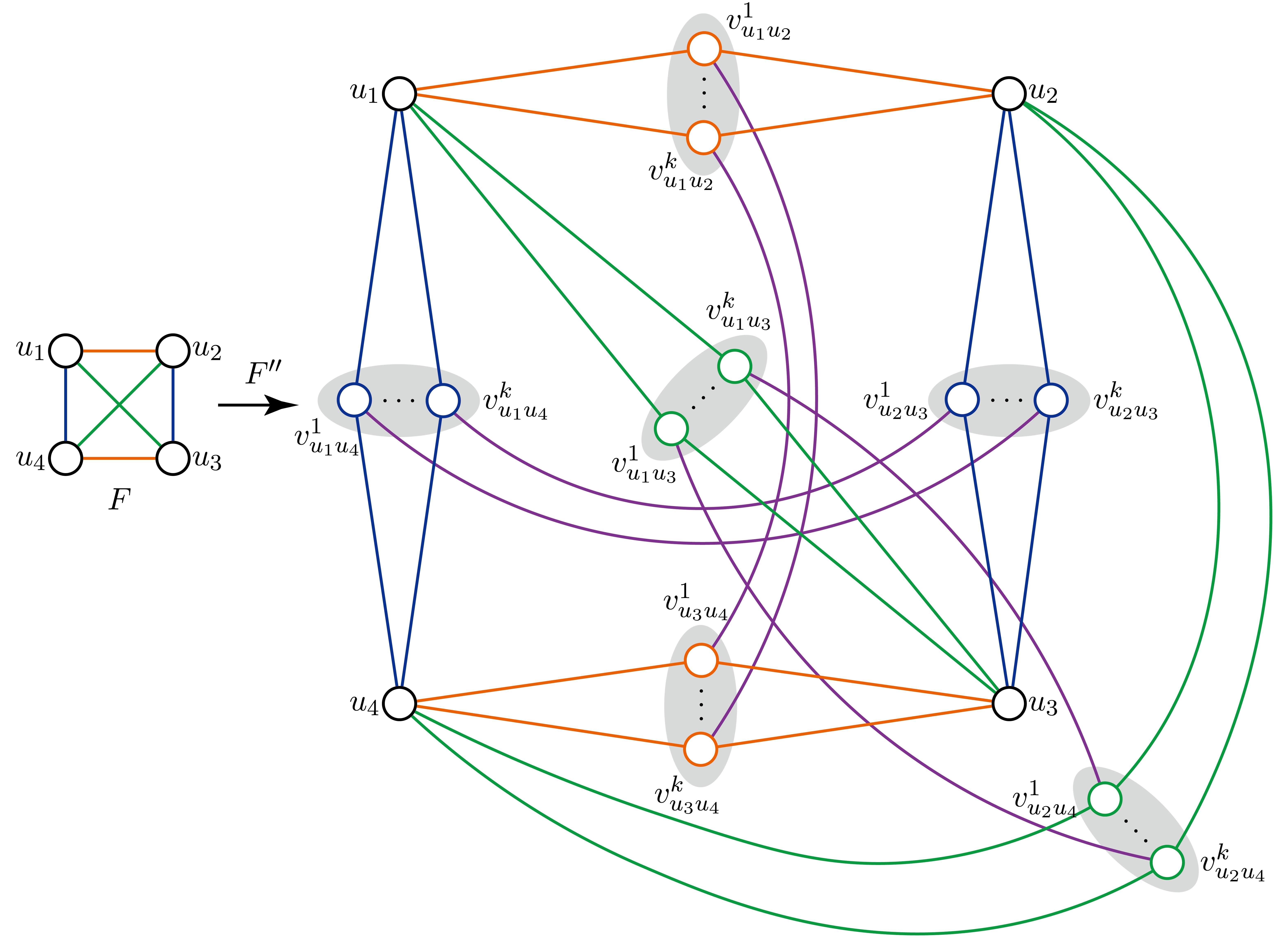}
	\caption{A sample of $F$ and $F''$ when $\ell=4$. The orange, blue, and green edges in $F$ correspond to $M_1, M_2,$ and $M_3$. The color of each part in $F''$ indicates where it comes from. Each grey part is a clique added in step 2. The purple edges are added in Step 3, i.e., in $E_2$.}\label{fig: lb flawless}
\end{figure}
Given $t_1,t_2\in\{1,\ldots,k\}$,
 then $\{v_{e_1}^{t_1}\}_{e_1\in M_i} \cup \{v_{e_2}^{t_2}\}_{e_2\in M_j} \cup V(F)$ forms an induced cycle of length $r$ for any $1\le i,j\le 3$ and $i\neq j$. There are $\binom{3}{2}\binom{k}{2} = (\frac{2}{3} +o(1)) \frac{n^2}{2d^2}$ possible choices and it is not hard to verify that the intersection of each two such cycles has size $\ell$ or $\ell+d$.
We say an induced copy of $C_r$ is \textbf{good} if it can be characterized by some $M_i, M_j$ and $t_1,t_2$ from the above way.
To prove $F''$ is $(C_r,L)$-intersecting, it suffices to prove that there are no other induced copies of $C_r$.

If we remove all edges from $E_1 \cup E_2$, then it results in a bipartite graph with one part $V(F)$ and the other part $V(F'')\setminus V(F)$. Note that all $C_r$ must contain all vertices in $V(F)$ since $|V(F)|=\ell$ and $r=2\ell$.
Then by the third rule of flawless $1$-factorization, it is  easy to verify that the cycle is good.

If there exists an induced copy of $C_r$~(denoted by $C$) contained some edge in $E_1$, say $v_{u_iu_j}^{t_1}v_{u_iu_j}^{t_2}$, then $u_iv_{u_iu_j}^{t_1}, u_iv_{u_iu_j}^{t_2},u_jv_{u_iu_j}^{t_1}, u_jv_{u_iu_j}^{t_2}\in E(F'')$. Since $C$ is an induced copy of $C_r$ and $r\ge 8$,  $u_i,u_j\notin V(C)$. Then the other edge connecting $v_{u_iu_j}^{t_1}$ in $C$ must be in $E_2$, say $v_{u_iu_j}^{t_1}v_{u_{i'}u_{j'}}^{t_3}$. We claim $t_3=t_2$; otherwise $\{v_{u_iu_j}^{t_1}, v_{u_iu_j}^{t_2},v_{u_{i'}u_{j'}}^{t_3}\}$ forms an triangle, a contradiction with $C$ being an induced copy of $C_r$. Similarly, the other neighbor of $v_{u_iu_j}^{t_2}$ in $C$ must be of the form $v_{u_{i''}v_{j''}}^{t_1}$. Then $v_{u_{i''}u_{j''}}^{t_1}v_{u_{i'}u_{j'}}^{t_2} \in E(F'')$ and $\{v_{u_iu_j}^{t_1}, v_{u_iu_j}^{t_2},v_{u_{i''}u_{j''}}^{t_1},v_{u_{i'}u_{j'}}^{t_2}\}$ forms a $C_4$, a contradiction  with $C$ being an induced copy of $C_r$.

Let $\gamma_i := \frac{1}{2}|V(C) \cap N_{F''}(u_i)| + \mathrm{1}_{u_i \in V(C)}$, $1\le i\le \ell$. Then $\gamma_i \le 2$ for all $1\le i\le \ell$. Note that $|V(C)| = r = 2\ell = \sum_{i=1}^{\ell}\gamma_i$ which implies $\gamma_i = 2$ for all $1\le i\le \ell$.
 If we have an edge $v_{u_iu_j}^{t_1}v_{u_{i'}u_{j'}}^{t_1} \in E_2 \cap E(C)$, then either $u_i \notin V(C)$ or $u_j\notin V(C)$ which implies either $\gamma_i \le 3/2$ or $\gamma_j \le 3/2$, a contradiction.
Therefore, all induced copies of $C_r$ in $F''$ are good and then we have
\[
\Psi_r(n,C_r,L) \ge \left(\frac{2}{3}+o(1)\right)\frac{n^2}{2d^2}.
\]

\subsection{The upper bound}\label{sec: ub 6}

Recall $L=\{\ell_1,\ell_2\}\subseteq [1,r-1]$, $\ell_1,\ell_2,r$ form an arithmetic progression with common difference $d\ge 2$, $\ell := \ell_1=2d$ and then $\ell_2 = \ell +d, r=\ell+2d=4d\ge 8$.
Let $G$ be the extremal $(C_r, L)$-intersecting graph with $n$ vertices maximizing $N_{ind}(G,C_r)$.
That is, $\Psi_r(n,C_r,L) = N_{ind}(G,C_r)$.
By the low bound given in Theorems~\ref{thm: the main theorem 2}~(2) and  \ref{thm: deza-E-F}, we can assume there is $W\subseteq \bigcap_{A \in \mathscr{H}^r_{G,C_r}}A  $ such that $|W|=\ell$.
By Lemma~\ref{lem: W not independent set psi},  we may assume $W = \{w_1,w_2,\ldots,w_{\ell}\}$ is an independent set.
Let $\mathcal{S}, X_1$ be the corresponding sets defined in Section~\ref{sec: atom}.

We  construct an auxiliary graph  $G'$, whose vertex set is $V(G')=\mathcal{S}$ and edge set is
$$E(G')=\bigg\{\{S_1,S_2\}\in \binom{\mathcal{S}}{2}:~S_1\cup S_2\subseteq A \text{~for some $A\in H^{r-\ell}_{G,C_r,W}$}\bigg\}.$$
Since $s=2$, by the definition of atoms, we have that $N_{ind}(G[X_1\cup W],C_r) = |E(G')|$.

\begin{lemma}\label{lem: s=2 l=2d}
    We have that
    \[N_{ind}(G[X_1\cup W],C_r) = |E(G')| \le  \left(1 - \frac{1}{f(\ell)}+o(1)\right)\frac{|\mathcal{S}|^2}{2}.\]
\end{lemma}

\begin{proof}
Let $w\in W$ and $S \in \mathcal{S}$ an atom with size $d$.
Define $\alpha(w,S) := |\{x \in S: xw\in E(G)\}|$. By the definition of the atom, $\alpha(w, S) \in \{0,1,2\}$.
Let $ W=\{ w_1,\ldots,w_{\ell}\}$. We partition $\mathcal{S}$ into three parts
\[
    \mathcal{S}_i = \{S \in \mathcal{S} : \alpha(w_1, S) = i\}, i=0,1,2.
\]
Note that the vertex set of $G'$ is $\mathcal{S}$ and we can view $\mathcal{S}_{i}~(i=0,1,2)$ as a vertex partition of $V(G')$.
We first claim that $\mathcal{S}_0$ is an independent set in $G'$, which directly follows from its definition.
Similarly, we have that $\mathcal{S}_2$ is an independent set, and there is no edge between $\mathcal{S}_1$ and $\mathcal{S}_0 \cup \mathcal{S}_2$.
Then $G'[\mathcal{S}_0 \cup \mathcal{S}_2]$ is a bipartite graph.
Thus  $|E(G')| \le |E(G'[\mathcal{S}_1])| + \frac{(|\mathcal{S}| - |\mathcal{S}_1|)^2 }{4}$.

Similarly, if we continue to partition $\mathcal{S}_1$ into three parts with the value of $\alpha(w_2, S)$, that is
\[
    \mathcal{S}_{1,i} = \{S \in \mathcal{S}_1 : \alpha(w_2, S) = i\}, i=0,1,2.
\]
By a similar argument, we can derive that $|E(G'[\mathcal{S}_1])| \le |E(G'[\mathcal{S}_{1,1}])| + \frac{(|\mathcal{S}_1| - |\mathcal{S}_{1,1}|)^2 }{4}$. Then we have that
\begin{equation*}
\begin{aligned}
    |E(G')| & \le |E(G'[\mathcal{S}_1])| + \frac{(|\mathcal{S}| - |\mathcal{S}_1|)^2 }{4} \\
    & \le |E(G'[\mathcal{S}_{1,1}])| + \frac{(|\mathcal{S}_1| - |\mathcal{S}_{1,1}|)^2 }{4} + \frac{(|\mathcal{S}| - |\mathcal{S}_1|)^2 }{4} \\
    & \le |E(G'[\mathcal{S}_{1,1}])| + \frac{(|\mathcal{S}| - |\mathcal{S}_{1,1}|)^2 }{4}.
\end{aligned}
\end{equation*}

Keep doing the partition with $\alpha(w_j, S)$ for $j=3,\ldots,\ell$.
 Let $\mathcal{S}' = \{S\in \mathcal{S} : \alpha(w_j,S)=1,\forall 1 \le j \le \ell \}$. Then we can derive that $|E(G')| \le |E(G'[\mathcal{S}'])| + \frac{(|\mathcal{S}| - |\mathcal{S}'|)^2 }{4}$.
To prove Lemma~\ref{lem: s=2 l=2d}, noting that $f(\ell) \ge 3$ for any $\ell$, it suffices to prove that
\[
|E(G'[\mathcal{S}'])| \le \left(1 - \frac{1}{f(\ell)} + o(1)\right)\frac{|\mathcal{S}'|^2}{2}.
\]


In the following, we assume for each $S \in \mathcal{S}'$, $d_{G'[\mathcal{S}']}(S) \ge (1-\frac{1}{f(\ell)})|\mathcal{S}'|$.
Otherwise, we can delete $S$ and then prove the result for $\mathcal{S'}\setminus\{S\}$.
Also, we can assume $|\mathcal{S}'| \ge 3f(\ell)$ since we have an $o(1)$ term in our result and $f(\ell)$ is a constant.

For an atom $S \in \mathcal{S}'$, let $\phi(S)$ be the graph with vertex set $\{w_1,w_2,\ldots,w_\ell\}$ and $w_iw_j \in E(\phi(S))$ if there exists $x \in S$ such that $xw_i,xw_j \in E(G)$.
In fact, the edge $w_iw_j \in E(\phi(S))$ comes from contracting the vertex $x$ which connects $w_i$ and $w_j$.
The vertex $x$ is denoted as $B(\phi(S), w_iw_j)$.
See Figure~\ref{fig: phi} for an example.

\begin{figure}[t]
	\centering
	\includegraphics[width=0.5\linewidth]{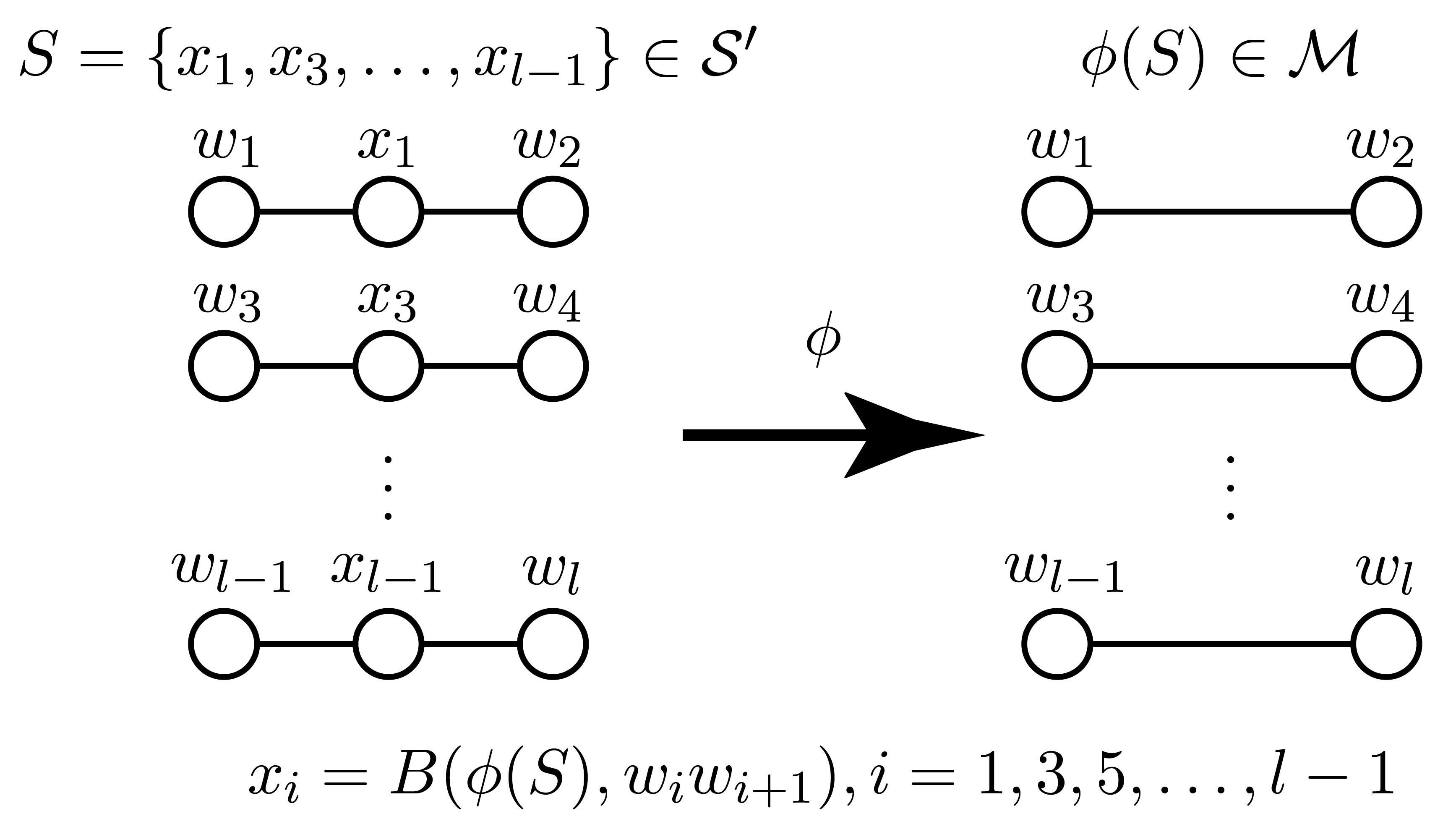}
	\caption{A sample of $\phi$. }\label{fig: phi}
\end{figure}

Note that $|S| = d$ for each $S \in \mathcal{S}'$.
Since $l=2d$ and $W$ is an independent set in $G$, by the definition of $\mathcal{S}'$, $\phi(S)=dK_2$ for any $S \in \mathcal{S}'$, where $dK_2$ means $d$ disjoint edges. Since $\ell = 2d$, $\phi(S)$ is a perfect matching on vertices $\{w_1,w_2,\ldots,w_\ell\}$.
Let $\mathcal{M}$ be the set of perfect matching on vertices $\{w_1,w_2,\ldots,w_l\}$. Then  $\phi(S) \in \mathcal{M}$ for any $ S \in \mathcal{S}'$.

\vspace{.2cm}
\begin{claim}\label{claim: K f(l)+1 free}
    $G'[\mathcal{S}']$ is $K_{f(\ell)+1}$-free.
\end{claim}

\noindent
\textbf{Proof of Claim~\ref{claim: K f(l)+1 free}}: Suppose there are $S_1,\ldots,S_{f(\ell)+1}\in \mathcal{S}'$ such that $G'[S_1,\ldots,S_{f(\ell)+1}]$ is a clique.
Then $G[S_i\cup S_j \cup W]$ is an induced copy of $C_r$ for $i,j\in \{1,\ldots,f(l)+1\}$ and $ i \neq j $.
Let $M_i = \phi(S_i),i=1,\ldots,f(l)+1$.
By the definitions of $G'$ and the graph $\phi$, we have  $M_i \cap  M_j = \emptyset$ for any $i,j\in \{1,\ldots,f(l)+1\}$ and $ i \neq j $.
Let $H$ denote the graph with vertex set $\{w_1,w_2,\ldots,w_\ell\}$ and edge set $\bigcup_{i=1}^{f(\ell)+1}E(M_i)$.
We now claim $\{M_1,\ldots,M_{f(\ell)+1}\}$ is a flawless $1$-factorization of $H$.

First, for each $i$, $M_i$ is a perfect matching.
Second, for each $i \neq j$, $M_i \cup M_j$ induces a Hamiltonian cycle of $H$ since $G[S_i\cup S_j \cup W]$ is an induced cycle.
The third part takes more effort to prove.
If there is a Hamiltonian cycle in $H$, then it corresponds to an induced $C_r$ in $G$.
For any $i_1,i_2 \in \{1,\ldots,f(\ell)+1\}$, let $e_1 \in E(M_{i_1}),e_2 \in E(M_{i_2})$ be two distinct edges from $E(H)$. Then $x_1: = B(M_{i_1}, e_1)$ and $x_2: = B(M_{i_2},e_2)$ are two vertices in $G$. We have that $x_1x_2 \notin E(G)$ due to the following two cases.
\begin{enumerate}
    \item  $i_1=i_2$. Then $x_1,x_2 \in S_{i_1}$. Since $S_{i_1} \cup W$ is contained in some induced $C_r$ in $G$, by the definition of  $\phi$, we have $x_1x_2 \notin E(G)$.
    \item $i_1 \neq i_2$. Since $G[S_{i_1} \cup S_{i_2} \cup W]$ is an induced copy of $C_r$,  by the definition of  $\phi$, we also have $x_1x_2 \notin E(G)$.
\end{enumerate}

Let $C=w_{i_1}w_{i_2}\ldots w_{i_\ell}w_{i_1}$ be a Hamiltonian cycle in $H$. Assume $w_{i_{j}}w_{i_{j+1}} \in E(M_{i_j'})$ for $j=1,2,\ldots,l$~(write $i_{\ell+1}=i_1$). Then
\[
w_{i_1} B(M_{i_1'}, w_{i_1}w_{i_2})w_{i_2}B(M_{i_2'}, w_{i_2}w_{i_3})\ldots B(M_{i_{\ell-1}'}, w_{i_{\ell-1}}w_{i_\ell}) w_{i_\ell} B(M_{i_\ell'}, w_{i_\ell}w_{i_1})w_{i_1}
\]
is an induced $C_r$ in $G[X_1\cup W]$ by the definition of  $\phi$,
 and we denote the corresponding induced cycle copy of $C_r$ in $G$  by $C_r(C)$.
If $C$ is not a union of some $M_i,M_j$, then there exists $M_{j'}$ such that $0 < |E(C) \cap E(M_{j'})| < d$ by $|E(C)| = \ell = 2d$.
By assumption, $|\mathcal{S}'| \ge 3f(\ell)$ and $d_{G'[\mathcal{S}']}(S_{j'}) \ge (1-\frac{1}{f(\ell)})|\mathcal{S}'| \ge 3(f(\ell)-1) \ge f(\ell)+2$. Then there exists $S_{j''}\in \mathcal{S}'$ which is adjacent to $S_{j'}$ in $G'$ and $S_{j''} \notin \{S_1,S_2,\ldots,S_{f(\ell)+1}\}$.
Then  $S_{j'}\cup S_{j''}\cup W$ also induces an induced $C_r$ in $G$ whose intersection with $C_r(C)$ has size $|W| + |E(H) \cap M_{j'}| \notin L=\{\ell, \ell+d\} $, a contradiction with $G$ being $(C_r,L)$-intersecting.

Therefore, $\{M_1,M_2,\ldots,M_{f(\ell)+1}\}$ is a flawless $1$-factorization of $H$ with $\ell$ vertices which contradicts the definition of $f(\ell)$.
\hfill $\blacksquare$\par

By Claim \ref{claim: K f(l)+1 free} and Tur\'an's Theorem~\cite{turan1941extremal}, $|E(G'[\mathcal{S}'])| \le (1-\frac{1}{f(\ell)})\binom{|\mathcal{S}'|}{2}$ which proves our lemma.
\end{proof}

Combining Lemma~\ref{lem: s=2 l=2d} and (\ref{eq: only consider X1}), we have that
\begin{equation*}
\begin{aligned}
    \Psi_r(n,C_r,L) & \le N_{ind}(G[X_1\cup W],C_r) + O(n) \\
    & \le \left(1 - \frac{1}{f(\ell)}+o(1)\right)\frac{|\mathcal{S}|^2}{2} + O(n) \\
    & \le \left(1 - \frac{1}{f(\ell)}+o(1)\right)\frac{n^2}{2d^2},
\end{aligned}
\end{equation*}
which proves the upper bound.

\section*{Acknowledgement}
This work is supported by the National Natural Science Foundation of China (Grant 12571372).
\section*{Declaration of competing interest}
The authors declare that they have no known competing financial interests or personal relationships that could have appeared to influence the work reported in this paper.

\section*{Data availability}
No data was used for the research described in the article.

\bibliography{ref.bib}
\bibliographystyle{wyc3}

\end{document}